\documentclass[12pt]{article}
\usepackage{amssymb,amsmath,amsthm, amsfonts}
\usepackage{graphicx}
\usepackage{epsfig}
\usepackage{tikz}
\usepackage{mathrsfs}

\def\disp{\displaystyle}

\def\dref#1{(\ref{#1})}

\theoremstyle{plain}
\newtheorem{theorem}{Theorem}[section]
\newtheorem{lemma}{Lemma}[section]

\theoremstyle{definition}
\newtheorem{definition}{Definition}[section]

\newtheorem{remark}{Remark}[section]

\numberwithin{equation}{section}

\begin{document}

\title{\bf Blow-up prevention by nonlinear diffusion in a 2D Keller-Segel-Navier-Stokes system with rotational flux}

\author{
Yuanyuan Ke$^{a}$, Jiashan Zheng$^{b,a}$,\thanks{Corresponding author. E-mail address: zhengjiashan2008@163.com (J. Zheng)}
\\
$^{a}$ School of Information,
\\
Renmin University of China, Beijing, 100872, P.R.China
\\
$^{b}$ School of Mathematics and Statistics Science,
\\
Ludong University, Yantai 264025, P.R.China
}
\date{}

\maketitle \vspace{0.3cm}
\noindent
\begin{abstract}
This paper investigates the following  Keller-Segel-Navier-Stokes system with  nonlinear diffusion
and  rotational flux
$$
\left\{\begin{array}{lll}
&n_t+u\cdot\nabla n=\Delta n^m-\nabla\cdot(nS(x, n, c)\nabla c),\quad &x\in \Omega, t>0,
\\
&c_t+u\cdot\nabla c=\Delta c-c+n,\quad &x\in \Omega, t>0,
\\
&u_t+\kappa (u\cdot\nabla)u+\nabla P=\Delta u+n\nabla \phi,\quad &x\in \Omega, t>0,
\\
&\nabla\cdot u=0,\quad &x\in \Omega, t>0,
\end{array}\right.\eqno(KSNF)
$$
where  $\kappa\in \mathbb{R},\phi\in W^{2,\infty}(\Omega)$ and $S$
is a given function with values in  $\mathbb{R}^{2\times2}$
which fulfills
$$
|S(x,n,c)| \leq C_S
$$
with some $C _S > 0$.
Systems of this type  describe chemotaxis-fluid interaction in cases when the evolution of the chemoattractant is essentially
dominated by production through cells.
If $m>1$ and $\Omega\subset \mathbb{R}^2$ is a {\bf bounded} domain with
smooth boundary, then for all reasonably regular initial data, a corresponding initial-boundary value
problem for $(KSNF)$ possesses  a global and bounded (weak) solution,
which significantly improves
previous results of several authors. Moreover, the {\bf optimal condition} on the parameter $m$ for  global existence is obtained.
Our approach underlying the derivation of main result
is based on an entropy-like estimate involving the functional
%Our main tool is consideration of the energy functional
$$\int_{\Omega}(n_{\varepsilon} +\varepsilon)^{m}+\int_{\Omega}|\nabla c_\varepsilon|^{2},$$
where $n_\varepsilon$ and $c_\varepsilon$ are components of the solutions to \dref{1.1fghyuisda} below.
\end{abstract}

\vspace{0.3cm}


\newpage
\section{Introduction}

Chemotaxis, the biased movement of cells in response to chemical gradients,
plays an important role in coordinating cell migration in many biological phenomena (see Hillen and Painter \cite{Hillen}).
For example, the fruit fly Drosophila melanogaster navigates up gradients of attractive odours during food location,
and male moths follow pheromone gradients released by the female during mate location.
In 1970 Keller and Segel \cite{Keller2710} proposed a mathematical model describing chemotactic aggregation of cellular slime molds.
But in their model, they did not take into account the relationship between cells and their environment.
So the model can be used to describe that bacterial chemotaxis was viewed as locomotion in an otherwise quiescent fluid.
Yet suspensions of aerobic bacteria often develop flows from the interplay of chemotaxis and buoyancy.
Tuval and his cooperator \cite{Tuval1215} described the above biological phenomena and
proposed the mathematical model consisting of oxygen diffusion and consumption,
chemotaxis, and fluid dynamics
$$
\left\{\begin{array}{lll}
&n_t+u\cdot\nabla n=\Delta n-\nabla\cdot(n\chi(c)\nabla c),\quad &x\in \Omega, t>0,
\\
&c_t+u\cdot\nabla c=\Delta c-nf(c),\quad &x\in \Omega, t>0,
\\
&u_t+\kappa (u\cdot\nabla)u+\nabla P=\Delta u+n\nabla \phi,\quad &x\in \Omega, t>0,
\\
&\nabla\cdot u=0,\quad &x\in \Omega, t>0
\end{array}\right.
$$
in a domain $\Omega\subset \mathbb{R}^N(N\geq1)$,
where $n$, $c$, $u$, and $P$ denote, respectively, the density of cells, chemical concentration, velocity field and pressure of the fluid.
The coefficient $\kappa$ is related to the strength of nonlinear fluid convection,
$\phi$ stands for the potential of the gravitational field within which the cells are driven through buoyant forces,
the function $\chi(c)$ measures the chemotactic sensitivity,
and $f(c)$ represents the oxygen consumption rate.
Some modeling approaches suggested that an adequate description of bacterial motion near surfaces of their surrounding fluid should
involve rotational components in the cross-diffusive flux (see \cite{Xuess1215, Xue1215}),
so the natural generalizations of chemotaxis-fluid systems should model the evolution of the cell density, as the following form
$$
\left\{\begin{array}{lll}
&n_t+u\cdot\nabla n=\Delta n-\nabla\cdot(nS(x, n, c)\nabla c),\quad &x\in \Omega, t>0,
\\
&c_t+u\cdot\nabla c=\Delta c-nf(c),\quad &x\in \Omega, t>0,
\\
&u_t+\kappa (u\cdot\nabla)u+\nabla P=\Delta u+n\nabla \phi,\quad &x\in \Omega, t>0,
\\
&\nabla\cdot u=0,\quad &x\in \Omega, t>0
\end{array}\right.
$$
where $S$ stands for the chemotactic sensitivity. Moreover,
since the diffusion of bacteria (or, more generally, of cells) in a viscous fluid is more like movement in a porous medium,
the authors in \cite{Francesco12186} extended the above model to one with a porous medium-type diffusion
$$
\left\{\begin{array}{lll}
&n_t+u\cdot\nabla n=\Delta n^m-\nabla\cdot(nS(x, n, c)\nabla c),\quad &x\in \Omega, t>0,
\\
&c_t+u\cdot\nabla c=\Delta c-nf(c),\quad &x\in \Omega, t>0,
\\
&u_t+\kappa (u\cdot\nabla)u+\nabla P=\Delta u+n\nabla \phi,\quad &x\in \Omega, t>0,
\\
&\nabla\cdot u=0,\quad &x\in \Omega, t>0,
\end{array}\right.
$$
where $m>1$.
Concerning the framework where the chemical is produced by the cells instead of consumed,
then the corresponding chemotaxis-fluid model is then the quasilinear Keller-Segel-Navier-Stokes system of the form (see \cite{Bellomo1216,Hillen})
\begin{equation}\label{0.4}
\left\{\begin{array}{lll}
&n_t+u\cdot\nabla n=\Delta n^m-\nabla\cdot(nS(x, n, c)\nabla c),\quad &x\in \Omega, t>0,
\\
&c_t+u\cdot\nabla c=\Delta c-c+n,\quad &x\in \Omega, t>0,
\\
&u_t+\kappa (u\cdot\nabla)u+\nabla P=\Delta u+n\nabla \phi,\quad &x\in \Omega, t>0,
\\
&\nabla\cdot u=0,\quad &x\in \Omega, t>0.
\end{array}\right.
\end{equation}

Due to the presence of the tensor-valued sensitivity as well as the strongly nonlinear term $(u\cdot\nabla)u$ and lower regularity for $n$,
the mathematical analysis of (\ref{0.4}) regarding global and bounded solutions is far from trivial.
Some simplified cases of the system (\ref{0.4}) have been studied.
When $\kappa=0$, which is corresponding to the chemotaxis-Stokes system, the results focused on the global existence and boundedness of the solutions,
for example, Wang and Xiang (\cite{Wang21215}) dealt with the case $m=1$ in $2$-dimensional space;
while for $m\not=1$, Li, Wang and Xiang (\cite{Liggghh793}),
Peng and Xiang (\cite{Peng55667}) considered the problem with the spatial dimension $N=2$ and $N=3$, respectively.
When $\kappa\not=0$, $m=1$ and $|S(x,n,c)|\le C_S(1+n)^{-\alpha}$ for some $C_S\ge0$ and $\alpha>0$,
Wang, Winkler and Xiang (\cite{Wang23421215}) and Ke and Zheng (\cite{Zhengssdddd00}) considered the global existence of the solution for the case $N=2$ and $N=3$, respectively.
But till now, as far as we know, it is still not clearly that in the case that $\kappa\not=0$ and $\alpha=0$, whether the solution of the system (\ref{0.4}) is bounded or not.
At the same time, we also noticed that when dealing with the problem of $\kappa=0$ and $\alpha=0$, or $\kappa\not=0$ and $\alpha>0$,
Li, Wang and Xiang (\cite{Liggghh793}) and Wang, Winkler and Xiang (\cite{Wang23421215}) both added the assumption that the domain is {\bf convex}.
Whether the convexity of the domain is necessary also arouses our interest.
By considering the key energy functional
$$
\int_{\Omega}n^{m} +\int_{\Omega}|\nabla c|^{2},
$$
we can obtain the global existence and boundedness of the solution for the system (\ref{0.4}),
which corresponding to the case that $\kappa\not=0$ and $\alpha=0$, in a more general {\bf non-convex} domain.

In this paper, we shall subsequently consider the chemotaxis-Navier-Stokes system (\ref{0.4}) along with the initial data
\begin{equation}\label{0.6}
\disp{n(x,0)=n_0(x),\quad c(x,0)=c_0(x),\quad u(x,0)=u_0(x),}\qquad x\in \Omega,
\end{equation}
and under the boundary conditions
\begin{equation}\label{0.5}
\disp{\left(nS(x, n, c)\nabla c\right)\cdot\nu=\nabla c\cdot\nu=0,\quad u=0,}\qquad x\in \partial\Omega, t>0,
\end{equation}
in a bounded domain $\Omega\subset \mathbb{R}^2$ with smooth boundary,
where we assume that the chemotactic sensitivity tensor $S(x, n, c)$ be satisfied
\begin{equation}\label{x1.73142vghf48rtgyhu}
S\in C^2(\bar{\Omega}\times[0,\infty)^2;\mathbb{R}^{2\times2})
\end{equation}
and
\begin{equation}\label{x1.73142vghf48gg}
|S(x, n, c)|\leq C_S ~~~~\mbox{for all}~~ (x, n, c)\in\Omega\times [0,\infty)^2
\end{equation}
with some $C_S > 0$.
Throughout this paper,
we assume that
\begin{equation}
\phi\in W^{2,\infty}(\Omega)
\label{dd1.1fghyuisdakkkllljjjkk}
\end{equation}
and the initial data $(n_0, c_0, u_0)$ fulfills
\begin{equation}\label{ccvvx1.731426677gg}
\left\{
\begin{array}{ll}
\displaystyle{n_0\in C^\kappa(\bar{\Omega})~~\mbox{for certain}~~ \kappa > 0~~ \mbox{with}~~ n_0\geq0 ~~\mbox{in}~~\Omega},
\\
\displaystyle{c_0\in W^{2,\infty}(\Omega)~~\mbox{with}~~c_0,w_0\geq0~~\mbox{in}~~\bar{\Omega},}
\\
\displaystyle{u_0\in D(A),}\\
\end{array}
\right.
\end{equation}
where $A$ denotes the Stokes operator with domain $D(A) := W^{2,{2}}(\Omega)\cap  W^{1,{2}}_0(\Omega)
\cap L^{2}_{\sigma}(\Omega)$,
and
$L^{2}_{\sigma}(\Omega) := \{\varphi\in  L^{2}(\Omega)|\nabla\cdot\varphi = 0\}$. (see \cite{Sohr}).

Within the above frameworks, our main result concerning global existence and boundedness of solutions to (\ref{0.4})-(\ref{0.5}) is as follows.

\begin{theorem}\label{theorem3}
Let $m>1$, $\Omega\subset \mathbb{R}^2$ be a bounded domain with smooth boundary, and assume \dref{x1.73142vghf48rtgyhu}-\dref{ccvvx1.731426677gg} hold.
Then the problem (\ref{0.4})-(\ref{0.5}) admits a global-in-time weak solution
$(n,c,u,P)$, which is uniformly bounded in the sense that
\begin{equation}
\|n(\cdot, t)\|_{L^\infty(\Omega)}+\|c(\cdot, t)\|_{W^{1,\infty}(\Omega)}+\|u(\cdot, t)\|_{L^{\infty}(\Omega)}\leq C~~ \mbox{for all}~~ t>0
\label{1.163072xggttyyu}
\end{equation}
with some positive constant $C$.
\end{theorem}

\begin{remark}
(i) If $u\equiv0$, Theorem \ref{theorem3} is (partly) coincides with Theorem 4.1 of \cite{Winkler79},
which is {\bf optimal} according to the fact that the 2D fluid-free system  admits a global bounded classical solution for $m>1$ as mentioned by \cite{Tao794} (see also \cite{Winkler79}).

(ii) Theorem \ref{theorem3} extends the results of Li, Wang and Xiang \cite{Liggghh793},
who proved the possibility of boundedness in the case that $\Omega$ is a bounded {\bf convex} domain $\Omega\subset \mathbb{R}^2$ with smooth boundary,
$\kappa=0$ and  $S$ satisfies \dref{x1.73142vghf48rtgyhu} as well as \dref{x1.73142vghf48gg} with some $m>1$.
\end{remark}

This paper is organized as follows.
In Section 2, we do some preliminary works and propose a approximate problem.
In Section 3, we use some iteration technique to establish the necessary a priori estimates.
Finally, in Section 4, we obtain the global existence and boundedness of the solutions for the system (\ref{0.4})-(\ref{0.5}) in a bounded domain.

\section{Preliminaries}

In order to construct a weak solutions by an approximation procedure, we construct the approximate problems as follows
\begin{equation}\label{1.1fghyuisda}
\left\{\begin{array}{lll}
&n_{\varepsilon t}+u_{\varepsilon}\cdot\nabla n_{\varepsilon}
=\Delta (n_{\varepsilon}+\varepsilon)^m-\nabla\cdot(n_{\varepsilon}S_\varepsilon(x, n_{\varepsilon}, c_{\varepsilon})\nabla c_{\varepsilon}),\quad
&x\in \Omega,\; t>0,
\\
&c_{\varepsilon t}+u_{\varepsilon}\cdot\nabla c_{\varepsilon}=\Delta c_{\varepsilon}-c_{\varepsilon}+n_{\varepsilon},\quad
&x\in \Omega,\; t>0,
\\
&u_{\varepsilon t}+\nabla P_{\varepsilon}=\Delta u_{\varepsilon}-\kappa (Y_{\varepsilon}u_{\varepsilon} \cdot \nabla)u_{\varepsilon}+n_{\varepsilon}\nabla \phi,\quad
&x\in \Omega,\; t>0,
\\
&\nabla\cdot u_{\varepsilon}=0,\quad
&x\in \Omega,\; t>0,
\\
&\disp{\nabla n_{\varepsilon}\cdot\nu=\nabla c_{\varepsilon}\cdot\nu=0,u_{\varepsilon}=0},\quad
&x\in \partial\Omega,\; t>0,
\\
&\disp{n_{\varepsilon}(x,0)=n_0(x),c_{\varepsilon}(x,0)=c_0(x),\;u_{\varepsilon}(x,0)=u_0(x)},\quad
&x\in \Omega,
\end{array}\right.
\end{equation}
where
$$
S_\varepsilon(x,n,c):= \rho_\varepsilon(x)\chi_\varepsilon(u)S(x, n, c),~~ x\in\bar{\Omega},~~n\geq0,~~c\geq0,
$$
$$
\rho_\varepsilon \in C^\infty_0 (\Omega)~~\mbox{such that}~~0\leq\rho_\varepsilon\leq1~~\mbox{in}~~\Omega~~\mbox{and}~~\rho_\varepsilon\nearrow1~~\mbox{in}~~\Omega~~\mbox{as}~~\varepsilon\searrow0,
$$
$$
\chi_\varepsilon \in C^\infty_0 ([0,\infty))~~\mbox{such that}~~0\leq\chi_\varepsilon\leq1~~\mbox{in}~~[0,\infty)~~\mbox{and}~~\chi_\varepsilon\nearrow1~~\mbox{in}~~[0,\infty)~~\mbox{as}~~\varepsilon\searrow0,
$$
and
$$
Y_{\varepsilon}w := (1 + \varepsilon A)^{-1}w ~\mbox{for all}~ w\in L^2_{\sigma}(\Omega)
$$
is a standard Yosida approximation.

By the well-established fixed-point arguments (see Lemma 2.1 of \cite{Winkler51215}, \cite{Winkler11215} and Lemma 2.1 of \cite{Painter55677}),
we could  show  the local solvability of system \dref{1.1fghyuisda}.

\begin{lemma}\label{lemma70}
Let $\Omega \subset \mathbb{R}^2$ be a bounded domain with smooth boundary, and assume \dref{x1.73142vghf48rtgyhu}-\dref{ccvvx1.731426677gg} hold.
For any $\varepsilon\in(0,1)$, there exist $T_{max,\varepsilon}\in (0,\infty]$ and
a classical solution $(n_\varepsilon, c_\varepsilon, u_\varepsilon, P_\varepsilon)$ of system \dref{1.1fghyuisda} in $\Omega\times[0,T_{max,\varepsilon})$.
Here
\begin{equation}\label{1.1ddfghyuisda}
\left\{\begin{array}{ll}
n_\varepsilon\in C^0(\bar{\Omega}\times[0,T_{max,\varepsilon}))\cap C^{2,1}(\bar{\Omega}\times(0,T_{max,\varepsilon})),
\\
c_\varepsilon\in  C^0(\bar{\Omega}\times[0,T_{max,\varepsilon}))\cap C^{2,1}(\bar{\Omega}\times(0,T_{max,\varepsilon}))\cap\bigcap_{p>1} L^\infty([0,T_{max,\varepsilon}); W^{1,p}(\Omega)),
\\
u_\varepsilon\in  C^0(\bar{\Omega}\times[0,T_{max,\varepsilon}))\cap C^{2,1}(\bar{\Omega}\times(0,T_{max,\varepsilon}))\cap \bigcap_{\gamma\in(0,1)}C^0([0,T_{max,\varepsilon}); D(A^\gamma)),
\\
P_\varepsilon\in  C^{1,0}(\bar{\Omega}\times(0,T_{max,\varepsilon})).
\end{array}\right.
\end{equation}
Moreover,  $n_\varepsilon$ and $c_\varepsilon$ are nonnegative in
$\Omega\times(0, T_{max,\varepsilon})$, and if $T_{max,\varepsilon}<+\infty$, then
$$
\limsup_{t\nearrow T_{max,\varepsilon}}[\|n_\varepsilon(\cdot, t)\|_{L^\infty(\Omega)}+\|c_\varepsilon(\cdot, t)\|_{W^{1,\infty}(\Omega)}+\|A^\gamma u_\varepsilon(\cdot, t)\|_{L^{2}(\Omega)}]=\infty
$$
for all $p > 2$ and $\gamma\in(\frac{1}{2}, 1)$.
\end{lemma}

\begin{lemma}(\cite{Tao41215})\label{lemma630}
Let $T\in(0,\infty]$, $\sigma\in(0,T)$, $A>0$ and $B>0$, and suppose that $y:[0,T)\rightarrow[0,\infty)$ is absolutely continuous and such that
$$
 y'(t)+Ay(t)\leq h(t) ~~\mbox{for a.e.}~~t\in(0,T)
$$
with some nonnegative function $h\in  L^1_{loc}([0, T))$ satisfying
$$
\int_{t}^{t+\sigma}h(s)ds\leq B~~\mbox{for all}~~t\in(0,T-\sigma).
$$
Then
$$
y(t)\leq \max\{y_0+B,\frac{B}{A\tau}+2B\}~~\mbox{for all}~~t\in(0,T).
$$
\end{lemma}

\section{Some basic priori estimates}

In order to establish the global solvability of system \dref{1.1fghyuisda}, in this section,
we plan to derive some estimates for the approximate system \dref{1.1fghyuisda}, which  plays  a significant role in obtaining the main result.
Let us first state two basic estimates on $n_{\varepsilon}$ and $c_{\varepsilon}$.

\begin{lemma}\label{fvfgfflemma45} (\cite{Zhengssdddd00})
The solution of \dref{1.1fghyuisda} satisfies
\begin{equation}
\int_{\Omega}{n_{\varepsilon}}= \int_{\Omega}{n_{0}}~~\mbox{for all}~~ t\in(0, T_{max,\varepsilon})
\label{ddfgczhhhh2.5ghju48cfg924ghyuji}
\end{equation}
as well as
$$
\int_{\Omega}{c_{\varepsilon}}\leq \max\{\int_{\Omega}{n_{0}},\int_{\Omega}{c_{0}}\}~~\mbox{for all}~~ t\in(0, T_{max,\varepsilon}).
$$
\end{lemma}

According to Lemma \ref{fvfgfflemma45}, we can obtain the following energy-type equality,
which was also used in Lemma 3.3 in \cite{Zhengssdddd00} (see also \cite{Zhenddddgssddsddfff00,Wang23421215}).

\begin{lemma}\label{lemmajddggmk43025xxhjklojjkkk}
Let $m>1$.
Then there exists $C>0$ independent of $\varepsilon$ such that the solution of \dref{1.1fghyuisda} satisfies
\begin{equation}
\disp{\int_{\Omega}n_{\varepsilon}+\int_{\Omega} (n_{\varepsilon}+\varepsilon)^{m-1}+\int_{\Omega} c_{\varepsilon}^2
+\int_{\Omega}  | {u_{\varepsilon}}|^2\leq C~~~\mbox{for all}~~ t\in (0, T_{max,\varepsilon}).}
\label{czfvgb2.5ghhjuyuccvviihjj}
\end{equation}
Moreover, for all $t\in(0, T_{max,\varepsilon}-\tau)$, it holds that one can find a constant $C > 0$ independent of $\varepsilon$ such that
\begin{equation}
\disp{\int_{t}^{t+\tau}\int_{\Omega} \left[  (n_{\varepsilon}+\varepsilon)^{2m-4} |\nabla {n_{\varepsilon}}|^2
+|\nabla {c_{\varepsilon}}|^2+ |\nabla {u_{\varepsilon}}|^2\right]\leq C,}
\label{bnmbncz2.5ghhjuyuivvbnnihjj}
\end{equation}
where $\tau=\min\{1,\frac{1}{6}T_{max,\varepsilon}\}.$
\end{lemma}

In order to obtain the boundedness of $n_{\varepsilon}$, we need to give higher norm estimates on $c_\varepsilon$.

\begin{lemma}\label{aasslemmafggg78630jklhhjj}
Let $(n_\varepsilon,c_\varepsilon,u_\varepsilon)$ be the solution of \dref{1.1ddfghyuisda} and $\tau=\min\{1,\frac{1}{6}T_{max,\varepsilon}\}$.
Then for any $q>2$, there exists $C: = C(q,K)$ independent of $\varepsilon$ such that
\begin{equation}
\|c_{\varepsilon}(\cdot, t)\|_{L^q(\Omega)}\leq C~~ \mbox{for all}~~ t\in(0,T_{max,\varepsilon}).
\label{3.10gghhjuuloollgghhhyhh}
\end{equation}
\end{lemma}

\begin{proof}
Let $p>3+4(m-1)$. Multiplying the second equation in $\dref{1.1fghyuisda}$ by ${c^{p-1}_{\varepsilon}}$, using the fact $\nabla\cdot u_{\varepsilon}=0$,
and applying the integration by parts, we have
\begin{equation}
\begin{array}{rl}
&\disp{\frac{1}{p}\frac{d}{dt}\int_{\Omega}c^{{{p}}}_{\varepsilon}+({{p}-1})
\int_{\Omega}c^{{{p}-2}}_{\varepsilon}|\nabla c_{\varepsilon}|^2+\int_{\Omega}c^{{{p}}}_{\varepsilon}}
\\
=&\disp{\int_\Omega c^{p-1}_{\varepsilon}n_{\varepsilon}}
\\
\leq&\disp{\int_\Omega c^{p-1}_{\varepsilon}(n_{\varepsilon}+\varepsilon)}
\\
\leq&\disp{\|n_{\varepsilon}+\varepsilon\|_{L^\frac{p-2(m-1)}{p-4(m-1)}(\Omega)}
\left(\int_\Omega c^{\frac{(p-1)[p-2(m-1)]}{{m-1}}}_{\varepsilon}\right)^{\frac{{m-1}}{p-2(m-1)}}
~~\mbox{for all}~~t\in(0,T_{max,\varepsilon})}
\end{array}
\label{3333cz2.5114114}
\end{equation}
by the H\"{o}lder inequality.
Now, due to the Gagliardo--Nirenberg inequality and \dref{ddfgczhhhh2.5ghju48cfg924ghyuji},  for some positive constants   $\kappa_0$ and $\kappa_1$, we derive
$$
\begin{array}{rl}
&\disp\left(\int_\Omega c^{\frac{(p-1)[p-2(m-1)]}{{m-1}}}_{\varepsilon}\right)^{\frac{{m-1}}{p-2(m-1)}}
\\
=&\disp{\|  c^{\frac{p}{2}}_{\varepsilon}\|^{\frac{2(p-1)}{p}}_{L^{\frac{(p-1)[p-2(m-1)]}{p(m-1)}}(\Omega)}}
\\
\leq&\disp{\kappa_{0}(\|\nabla c^{\frac{p}{2}}_{\varepsilon}\|_{L^2(\Omega)}^{\frac{p[p-2(m-1)-1]}{[p-1][p-2(m-1)]}}\| c^{\frac{p}{2}}_{\varepsilon}\|_{L^{\frac{2}{p}}(\Omega)}^{\frac{{2(m-1)}}{(p-1)[p-2(m-1)]}}
+\|c^{\frac{p}{2}}_{\varepsilon}\|_{L^\frac{2}{p}(\Omega)})^{\frac{2(p-1)}{p}}}\\
\leq&\disp{\kappa_{1}(\|\nabla c^{\frac{p}{2}}_{\varepsilon}\|_{L^2(\Omega)}^{\frac{2[p-2(m-1)-1]}{p-2(m-1)}}+1)}.
\end{array}
$$
So that, in light of \dref{3333cz2.5114114} and the Young inequality, we derive that for all $t\in(0,T_{max,\varepsilon})$,
$$
\begin{array}{rl}
&\disp{\frac{1}{p}\frac{d}{dt}\int_{\Omega}c^{{{p}}}_{\varepsilon}+({{p}-1})\int_{\Omega}c^{{{p}-2}}_{\varepsilon}|\nabla c_{\varepsilon}|^2+\int_{\Omega}c^{{{p}}}_{\varepsilon}}
\\
\leq&\disp{\kappa_{1}\|n_{\varepsilon}+\varepsilon\|_{L^\frac{p-2(m-1)}{p-4(m-1)}(\Omega)}(\|\nabla   c^{\frac{p}{2}}_{\varepsilon}\|_{L^2(\Omega)}^{\frac{2[p-2(m-1)-1]}{p-2(m-1)}}+1)}
\\
\leq&\disp{\frac{({{p}-1})}{2}\int_{\Omega}c^{{{p}-2}}_{\varepsilon}|\nabla c_{\varepsilon}|^2
+C_1(p)\kappa_{1}^{p-2(m-1)}\|n_{\varepsilon}+\varepsilon\|_{L^\frac{p-2(m-1)}{p-4(m-1)}(\Omega)}^{p-2(m-1)}+\kappa_{1}\|n_{\varepsilon}+\varepsilon
\|_{L^\frac{p-2(m-1)}{p-4(m-1)}(\Omega)},}\\
\end{array}
$$
where we have used the fact that $\frac{p-2(m-1)-1}{p-2(m-1)}+\frac{1}{p-2(m-1)}=1.$
In view of  $p>3+4(m-1)$, again, from the Young inequality, there exist positive constants $C_3$ and $C_4$ such that
\begin{equation}
\begin{array}{rl}
&\disp{\frac{1}{p}\frac{d}{dt}\int_{\Omega}c^{{{p}}}_{\varepsilon}+\frac{({{p}-1})}{2}\int_{\Omega}c^{{{p}-2}}_{\varepsilon}|\nabla c_{\varepsilon}|^2+\int_{\Omega}c^{{{p}}}_{\varepsilon}}
\\
\leq&\disp{C_2\|n_{\varepsilon}+\varepsilon\|_{L^\frac{p-2(m-1)}{p-4(m-1)}(\Omega)}^{p-2(m-1)}+C_3~\mbox{for all}~
t\in(0,T_{max,\varepsilon}).}\\
\end{array}
\label{3333cz2.51fggtyuujkkklii14114}
\end{equation}
In the following, we will estimate the integrals on the right-hand side of \dref{3333cz2.51fggtyuujkkklii14114}.
In view of the Gagliardo-Nirenberg inequality, for some $C_4,C_5$ and $C_6> 0$ which are independent of $\varepsilon$, we may derive from \dref{bnmbncz2.5ghhjuyuivvbnnihjj} that
$$
\begin{array}{rl}
&\disp\int_{t}^{t+\tau}\left(\|n_{\varepsilon}+\varepsilon\|_{L^\frac{p-2(m-1)}{p-4(m-1)}(\Omega)}^{p-2(m-1)}+C_3\right)ds
\\
=&\disp{\int_{t}^{t+\tau}\left(\|  (n_{\varepsilon}+\varepsilon)^{m-1}\|^{\frac{p-2(m-1)}{m-1}}_{L^{\frac{p-2(m-1)}{[p-4(m-1)](m-1)}}(\Omega)}+C_3\right)ds}
\\
\leq&\disp{C_{4}\int_{t}^{t+\tau}\left(\| \nabla{ (n_{\varepsilon}+\varepsilon)^{m-1}}\|^{2}_{L^{2}(\Omega)}\|{ (n_{\varepsilon}+\varepsilon)^{m-1}}\|^{{\frac{p}{m-1}}}_{L^{\frac{1}{m-1}}(\Omega)}+
\|{ (n_{\varepsilon}+\varepsilon)^{m-1}}\|^{\frac{p-2(m-1)}{m-1}}_{L^{\frac{1}{m-1}}(\Omega)}\right)+C_3}
\\
\leq&\disp{C_{5}\int_{t}^{t+\tau}\left(\| \nabla{ (n_{\varepsilon}+\varepsilon)^{m-1}}\|^{2}_{L^{2}(\Omega)}\right)+C_3}
\\
\leq&\disp{C_{6}},
\end{array}
$$
where $\tau=\min\{1,\frac{1}{6}T_{max,\varepsilon}\}.$
Therefore, \dref{3.10gghhjuuloollgghhhyhh} holds by applying Lemma \ref{lemma630} and the H\"{o}lder inequality.
\end{proof}

Based on Lemma \ref{lemmajddggmk43025xxhjklojjkkk} and Lemma \ref{aasslemmafggg78630jklhhjj},
we can get a series of important estimates of $n_\varepsilon$ and $c_\varepsilon$.

\begin{lemma}\label{lemma4556664ddd5630223}
Let $m>1$. Then the solution of \dref{1.1fghyuisda} satisfies
\begin{equation}
\int_{\Omega}(n_\varepsilon+\varepsilon)^{m}+\int_{\Omega}|\nabla c_{\varepsilon}|^{2}
\leq C ~~~\mbox{for all}~~ t\in(0,T_{max,\varepsilon})~~\mbox{and any}~~\varepsilon>0
\label{334444zjscz2.5297x9630222ssdd2114}
\end{equation}
and
\begin{equation}\int_{t}^{t+\tau}\int_{\Omega} (n_\varepsilon+\varepsilon) ^{2m}
\leq C~~ \mbox{for all}~~ t\in(0,T_{max,\varepsilon}-\tau)~~\mbox{and any}~~\varepsilon>0,
\label{3.10gghhjuuloollsdffffffgghhhy}
\end{equation}
where $\tau=\min\{1,\frac{1}{6}T_{max,\varepsilon}\}.$
\end{lemma}

\begin{proof}
Multiplying the first equation of $\dref{1.1fghyuisda}$ by ${(n_{\varepsilon}+\varepsilon)^{m-1}}$, integrating the product in $\Omega$,
and noticing $\nabla\cdot u_\varepsilon=0$, one obtains
$$
\begin{array}{rl}
&\disp{\frac{1}{{m}}\frac{d}{dt}\|n_\varepsilon+\varepsilon\|^{{m}}_{L^{{m}}(\Omega)}
+({m-1})\int_{\Omega} (n_\varepsilon+\varepsilon)^{{{2{m-3}}}}|\nabla n_\varepsilon|^2}
\\
=&\disp{-\int_\Omega  (n_{\varepsilon}+\varepsilon)^{m-1}\nabla\cdot(n_\varepsilon S_\varepsilon(x, n_{\varepsilon}, c_{\varepsilon})
\nabla c_\varepsilon) }
\\
=&\disp{ ({m-1}) \int_\Omega   (n_{\varepsilon}+\varepsilon)  ^{{m-2}} n_\varepsilon S_\varepsilon(x, n_{\varepsilon}, c_{\varepsilon})
\nabla n_\varepsilon\cdot\nabla c_\varepsilon}
\\
\leq&\disp{ C_S({m-1}) \int_\Omega   (n_{\varepsilon}+\varepsilon)  ^{m-1}
|\nabla n_\varepsilon||\nabla c_\varepsilon|~~\mbox{for all}~~ t\in(0,T_{max,\varepsilon})}
\end{array}
$$
by using \dref{x1.73142vghf48gg}. Then, by using the Young inequality, we have
\begin{equation}
\begin{array}{rl}
&\disp{\frac{1}{{m}}\frac{d}{dt}\|n_\varepsilon +\varepsilon \|^{{{m}}}_{L^{{m}}(\Omega)}
+({m-1})\int_{\Omega} (n_{\varepsilon}+\varepsilon)^{{{{2m-3}}}}|\nabla n_\varepsilon|^2}
\\
\leq&\disp{\frac{{m-1}}{2}\int_{\Omega} (n_{\varepsilon}+\varepsilon)^{{{{2m-3}}}}|\nabla n_\varepsilon|^2
+\frac{(m-1)C_S^2}{2}\int_\Omega (n_{\varepsilon}+\varepsilon)
|\nabla c_\varepsilon|^2~~\mbox{for all}~~ t\in(0,T_{max,\varepsilon}).}
\end{array}
\label{3333cz2.5kkssss1214114114}
\end{equation}
On the other hand,  in view of Lemma \ref{lemmajddggmk43025xxhjklojjkkk} and invoking the Gagliardo--Nirenberg inequality,
we infer with some $\gamma_{0}> 0$ and $\gamma_{1} > 0$ that
$$
\begin{array}{rl}
&\disp\int_\Omega (n_\varepsilon+\varepsilon)  ^{2{m}}
\\
=& \| (n_\varepsilon+\varepsilon) ^{\frac{2{m}-1}{2}}\|_{L^{\frac{4m}{2m-1}}(\Omega)}^{\frac{4m}{2m-1}}
\\
\leq& \gamma_{0}(\|\nabla (n_\varepsilon+\varepsilon)^{\frac{2{m}-1}{2}}
\|_{L^{2}(\Omega)}^{\frac{2{m}-1}{2m}} \|(n_\varepsilon+\varepsilon)^{\frac{2{m}-1}{2}}\|_{L^{\frac{2}{2{m}-1}}(\Omega)}^{\frac{1}{2m}}
+\| (n_\varepsilon+\varepsilon) ^{\frac{2{m}-1}{2}}\|_{L^{\frac{2}{2{m}-1}}(\Omega)})^{\frac{4m}{2m-1}}
\\
\leq&\gamma_{1}\| \nabla (n_\varepsilon+\varepsilon)  ^{\frac{2{m}-1}{2}}\|_{L^{2}(\Omega)}^{2}+\gamma_{1}.
\end{array}
$$
We then achieve, with the help of the above inequality, that
\begin{equation}\label{3333cz2.5kke345ddfff677ddff89001214114114rrggjjkk}
\begin{array}{rl}
&m{(m-1)}\disp\int_{\Omega} (n_\varepsilon+\varepsilon)  ^{{{{2m-3}}}}|\nabla n_\varepsilon|^2
\\
=&\disp\frac{4m{(m-1)}}{(2{m}-1)^2} \|\nabla (n_\varepsilon+\varepsilon)^{\frac{2{m}-1}{2}}\|_{L^{2}(\Omega)}^{2}
\\
\geq& \frac{1}{\gamma_{1}}\frac{4m{(m-1)}}{(2{m}-1)^2}(\disp\int_\Omega (n_\varepsilon+\varepsilon)  ^{2{m}}-1).
\end{array}
\end{equation}
Here, the Young inequality allows to be written as
$$
\begin{array}{rl}
&\disp{\frac{(m-1)C_S^2}{2}\int_\Omega (n_{\varepsilon}+\varepsilon)|\nabla c_\varepsilon|^2}
\\
 \leq&\disp{\varepsilon_1\int_\Omega (n_\varepsilon+\varepsilon)^{2m}+C_1(\varepsilon_1)\int_\Omega  |\nabla c_\varepsilon|^{\frac{4m}{2m-1}} ,}
\end{array}
$$
where
\begin{equation}\varepsilon_1=\frac{1}{\gamma_{1}}\frac{{m-1}}{(2{m}-1)^2}
\label{3333cz2hjjjj.563ss011228ddff}
\end{equation}
and
$$
C_1(\varepsilon_1)=\frac{2{m}-1}{2{m}}\left(\varepsilon_12m\right)^{-\frac{1}{2{m}-1} }\left(\frac{(m-1)C_S^2}{2}\right)^{\frac{2{m}}{2{m}-1} }.
$$
In light of \dref{3.10gghhjuuloollgghhhyhh}, there exist positive constants $l_0>\frac{1}{m-1}$ and $C_2$, such that
\begin{equation}\|c_\varepsilon(\cdot, t)\|_{L^{l_0}(\Omega)}\leq C_2~~ \mbox{for all}~~ t\in(0,T_{max,\varepsilon}).
\label{3.10gghhjukklllkklllokkllffghhjjoppuloollgghhhyhh}
\end{equation}
Next, with the help of the Gagliardo--Nirenberg inequality and \dref{3.10gghhjukklllkklllokkllffghhjjoppuloollgghhhyhh}, we derive that
$$
\begin{array}{rl}
&\disp{C_1(\varepsilon_1)\int_\Omega  |\nabla c_\varepsilon|^{\frac{4m}{2m-1}}}
\\
\leq&\disp{C_3\|\Delta c_\varepsilon\|_{L^{2}(\Omega)}^{a\frac{4m}{2m-1}}\| c_\varepsilon\|_{L^{l_0}(\Omega)}^{(1-a)\frac{4m}{2m-1}}
+C_3\| c_\varepsilon\|_{L^{l_0}(\Omega)}^{\frac{4m}{2m-1}}}
\\
\leq&\disp{C_{4}\|\Delta c_\varepsilon\|_{L^{2}(\Omega)}^{a\frac{4m}{2m-1}}+C_{4}}
\end{array}
$$
with some positive constants $C_3$ and $C_{4}$, where
$$a=\frac{\frac{1}{2}+\frac{1}{l_0}-\frac{2{m}-1}{4m}}{\frac{1}{2}+\frac{1}{l_0}}\in(0,1).$$
This, together with the Young inequality and $a\frac{4m}{2m-1}<2$ (due to $l_0>\frac{1}{m-1}$), yields
\begin{equation}
\label{ssdd3333cz2.5kkett677ddff734567789999001214114114rrggjjkk}
\disp{C_1(\varepsilon_1)\int_\Omega  |\nabla c_\varepsilon|^{\frac{4m}{2m-1}}\leq\frac{1}{4}\|\Delta c_\varepsilon\|_{L^{2}(\Omega)}^{2}+C_{5}.}
\end{equation}
Taking $-\Delta{c_{\varepsilon}}$ as the test function for the second  equation of \dref{1.1fghyuisda}, and using the Young inequality,
it yields that for all $t\in(0,T_{max,\varepsilon})$
\begin{equation}
\begin{array}{rl}
&\disp\frac{1}{{2}}\disp\frac{d}{dt}\|\nabla{c_{\varepsilon}}\|^{{{2}}}_{L^{{2}}(\Omega)}+
\int_{\Omega} |\Delta c_{\varepsilon}|^2+ \int_{\Omega} | \nabla c_{\varepsilon}|^2
\\
=&\disp{-\int_{\Omega} n_{\varepsilon}\Delta c_{\varepsilon}+\int_{\Omega} (u_{\varepsilon}\cdot\nabla c_{\varepsilon})\Delta c_{\varepsilon}}
\\
=&\disp{-\int_{\Omega} n_{\varepsilon}\Delta c_{\varepsilon}-\int_{\Omega}\nabla c_{\varepsilon}\nabla (u_{\varepsilon}\cdot\nabla c_{\varepsilon})}
\\
=&\disp{-\int_{\Omega} n_{\varepsilon}\Delta c_{\varepsilon}-\int_{\Omega}\nabla c_{\varepsilon}\nabla (\nabla u_{\varepsilon}\cdot\nabla c_{\varepsilon}),}
\end{array}
\label{hhxxcsssdfvvjjczddfdddfff2.5}
\end{equation}
where we have used the fact that
$$
\disp{\int_{\Omega}\nabla c_{\varepsilon}\cdot(D^2 c_{\varepsilon}\cdot u_{\varepsilon})
=\frac{1}{2}\int_{\Omega}  u_{\varepsilon}\cdot\nabla|\nabla c_{\varepsilon}|^2=0
~~\mbox{for all}~~ t\in(0,T_{max,\varepsilon}).}
$$
Meanwhile, we can further use Gagliardo-Nirenberg inequality and the elliptic regularity (\cite{Gilbarg4441215}) to conclude that for some $C_{6}> 0$,
$$
\disp \|\nabla c_{\varepsilon}\|_{L^{4}(\Omega)}^2\leq\disp{C_{6}\|\Delta c_{\varepsilon}\|_{L^{2}(\Omega)}\|\nabla c_{\varepsilon}\|_{L^{2}(\Omega)}
~~\mbox{for all}~~ t\in(0,T_{max,\varepsilon}).}
$$
This, together with the Cauchy-Schwarz inequality and the  Young  inequality, yields
\begin{equation}
\begin{array}{rl}
&\disp-\int_{\Omega}\nabla c_{\varepsilon}\nabla (\nabla u_{\varepsilon}\cdot\nabla c_{\varepsilon})
\\
\leq&\disp{\|\nabla u_{\varepsilon}\|_{L^{2}(\Omega)}\|\nabla c_{\varepsilon}\|_{L^{4}(\Omega)}^2}
\\
\leq&\disp{C_{6}\|\nabla u_{\varepsilon}\|_{L^{2}(\Omega)}\|\Delta c_{\varepsilon}\|_{L^{2}(\Omega)}\|\nabla c_{\varepsilon}\|_{L^{2}(\Omega)}}
\\
\leq&\disp{C_{6}^2\|\nabla u_{\varepsilon}\|_{L^{2}(\Omega)}^2\|\nabla c_{\varepsilon}\|_{L^{2}(\Omega)}^2
+\frac{1}{4}\|\Delta c_{\varepsilon}\|_{L^{2}(\Omega)}^2~~\mbox{for all}~~ t\in(0,T_{max,\varepsilon}).}
\end{array}
\label{hhxxcsssdfvvjjcddffzddfdddfff2.5}
\end{equation}
Applying the Cauchy-Schwarz inequality, one obtain
\begin{equation}
\begin{array}{rl}
\disp-\int_{\Omega} n_{\varepsilon}\Delta c_{\varepsilon}\leq&\disp{\frac{1}{4}
\int_{\Omega}|\Delta c_{\varepsilon}|^2+\int_{\Omega}n_{\varepsilon}^2~~\mbox{for all}~~ t\in(0,T_{max,\varepsilon})}.
\end{array}
\label{hhxxcsssdfvvjjddddczddfdddfff2.5}
\end{equation}
From \dref{hhxxcsssdfvvjjczddfdddfff2.5} and \dref{hhxxcsssdfvvjjcddffzddfdddfff2.5}, we thus infer that
\begin{equation}
\begin{array}{rl}
\disp\disp\frac{d}{dt}\|\nabla{c_{\varepsilon}}\|^{{{2}}}_{L^{{2}}(\Omega)}+
\int_{\Omega} |\Delta c_{\varepsilon}|^2+ 2\int_{\Omega} | \nabla c_{\varepsilon}|^2\leq&\disp{2\int_{\Omega}n_{\varepsilon}^2+
2C_{6}^2\|\nabla u_{\varepsilon}\|_{L^{2}(\Omega)}^2\|\nabla c_{\varepsilon}\|_{L^{2}(\Omega)}^2.}
\end{array}
\label{hhxxcsssdfvvsssjjczddfdddfff2.5}
\end{equation}
Collecting \dref{3333cz2.5kkssss1214114114}, \dref{ssdd3333cz2.5kkett677ddff734567789999001214114114rrggjjkk}--\dref{hhxxcsssdfvvsssjjczddfdddfff2.5}, we derive that for all $t\in(0,T_{max,\varepsilon})$,
$$
\begin{array}{rl}
&\disp{\frac{d}{dt}(\|n_\varepsilon+\varepsilon  \|^{{{m}}}_{L^{{m}}(\Omega)}+
\disp\|\nabla{c_{\varepsilon}}\|^{{{2}}}_{L^{{2}}(\Omega)})+
{m(m-1)}\int_{\Omega} (n_\varepsilon+\varepsilon)  ^{{{{2m-3}}}}|\nabla n_\varepsilon|^2}
\\
&+\disp{\frac{1}{{2}}
\int_{\Omega} |\Delta c_{\varepsilon}|^2+2 \int_{\Omega} | \nabla c_{\varepsilon}|^2}\\
\leq&\disp{m\varepsilon_1\int_\Omega   (n_\varepsilon+\varepsilon)  ^{2{m}}+2\int_{\Omega}n_{\varepsilon}^2+
2C_{6}^2\|\nabla u_{\varepsilon}\|_{L^{2}(\Omega)}^2\|\nabla c_{\varepsilon}\|_{L^{2}(\Omega)}^2+C_{7},}\\
\leq&\disp{m\varepsilon_1\int_\Omega   (n_\varepsilon+\varepsilon)  ^{2{m}}+2\int_{\Omega}(n_\varepsilon+\varepsilon)^2+
2C_{6}^2\|\nabla u_{\varepsilon}\|_{L^{2}(\Omega)}^2\|\nabla c_{\varepsilon}\|_{L^{2}(\Omega)}^2+C_{7}.}\\
\end{array}
$$
Moreover, it follows from the Young inequality and $m>1$, that
\begin{equation}
\begin{array}{rl}
&\disp{\frac{d}{dt}(\|n_\varepsilon +\varepsilon \|^{{{m}}}_{L^{{m}}(\Omega)}+
\disp\|\nabla{c_{\varepsilon}}\|^{{{2}}}_{L^{{2}}(\Omega)})+
{m(m-1)}\int_{\Omega} (n_\varepsilon+\varepsilon)  ^{{{{2m-3}}}}|\nabla n_\varepsilon|^2}
\\
&+\disp{\frac{1}{{2}}
\int_{\Omega} |\Delta c_{\varepsilon}|^2+2 \int_{\Omega} | \nabla c_{\varepsilon}|^2}\\
\leq&\disp{2m\varepsilon_1\int_\Omega   (n_\varepsilon+\varepsilon)  ^{2{m}}+
2C_{6}^2\|\nabla u_{\varepsilon}\|_{L^{2}(\Omega)}^2\|\nabla c_{\varepsilon}\|_{L^{2}(\Omega)}^2+C_{8}~~\mbox{for all}~~ t\in(0,T_{max,\varepsilon})}.
\end{array}
\label{3333cz2.5kksssssss121hhjjj4114114}
\end{equation}
By substituting \dref{3333cz2.5kke345ddfff677ddff89001214114114rrggjjkk} into \dref{3333cz2.5kksssssss121hhjjj4114114} and using \dref{3333cz2hjjjj.563ss011228ddff}, we find that
$$
\begin{array}{rl}
&\disp{\frac{d}{dt}(\|n_\varepsilon+\varepsilon  \|^{{{m}}}_{L^{{m}}(\Omega)}+
\|\nabla{c_{\varepsilon}}\|^{{{2}}}_{L^{{2}}(\Omega)})+
(\frac{1}{\gamma_{1}}\frac{4m{(m-1)}}{(2{m}-1)^2}-2{m}\varepsilon_1)\int_\Omega   (n_\varepsilon+\varepsilon)  ^{2{m}}}
\\
&+\disp{\frac{1}{{2}}
\int_{\Omega} |\Delta c_{\varepsilon}|^2+ 2\int_{\Omega} | \nabla c_{\varepsilon}|^2}\\
=&\disp{\frac{d}{dt}(\|n_\varepsilon+\varepsilon  \|^{{{m}}}_{L^{{m}}(\Omega)}+
\|\nabla{c_{\varepsilon}}\|^{{{2}}}_{L^{{2}}(\Omega)})+
\frac{1}{\gamma_{1}}\frac{2{m}{(m-1)}}{(2{m}-1)^2}\int_\Omega   (n_\varepsilon+\varepsilon)  ^{2{m}}}
\\
&+\disp{
\frac{1}{{2}}
\int_{\Omega} |\Delta c_{\varepsilon}|^2+ 2\int_{\Omega} | \nabla c_{\varepsilon}|^2}\\
\leq&\disp{
2C_{6}^2\|\nabla u_{\varepsilon}\|_{L^{2}(\Omega)}^2\|\nabla c_{\varepsilon}\|_{L^{2}(\Omega)}^2+C_{9}~~\mbox{for all}~~ t\in(0,T_{max,\varepsilon})}.
\end{array}
$$
Therefore, we derive from the Young inequality that
\begin{equation}
\begin{array}{rl}
&\disp{\frac{d}{dt}(\|n_\varepsilon+\varepsilon  \|^{{{m}}}_{L^{{m}}(\Omega)}+
\|\nabla{c_{\varepsilon}}\|^{{{2}}}_{L^{{2}}(\Omega)})+
2\int_\Omega   n_\varepsilon  ^{{{m}}}+ 2\int_{\Omega} | \nabla c_{\varepsilon}|^2+\frac{1}{\gamma_{1}}\frac{{m}({m-1})}{(2{m}-1)^2}\int_\Omega   (n_\varepsilon+\varepsilon)  ^{2{m}}}\\
\leq&\disp{
2C_{6}^2\|\nabla u_{\varepsilon}\|_{L^{2}(\Omega)}^2\|\nabla c_{\varepsilon}\|_{L^{2}(\Omega)}^2+C_{10}}\\
 \leq&\disp{
2C_{6}^2\|\nabla u_{\varepsilon}\|_{L^{2}(\Omega)}^2(\|\nabla c_{\varepsilon}\|_{L^{2}(\Omega)}^2+
\|n_\varepsilon+\varepsilon  \|^{{{m}}}_{L^{{m}}(\Omega)})+C_{10}~~\mbox{for all}~~ t\in(0,T_{max,\varepsilon}),}
\end{array}
\label{3333cz2.5kksssssss121hhjjj4sddfffdfff114114}
\end{equation}
where we have used the fact that $2\int_\Omega n_\varepsilon^{{{m}}}\leq \frac{1}{\gamma_{1}}\frac{{m}({m-1})}{(2{m}-1)^2}\int_\Omega (n_\varepsilon+\varepsilon)^{2{m}}+C_{10}$,
$m>1$ and the Young inequality.
Now, again, from the Gagliardo--Nirenberg inequality, \dref{bnmbncz2.5ghhjuyuivvbnnihjj}, and Lemma \ref{lemmajddggmk43025xxhjklojjkkk},
there exist constants $\gamma_{3}> 0$ and $\gamma_{4} > 0$, such that
\begin{equation}
\label{3333cz2.5kke345677ddff89001ddff214114114rrggjjkk}
\begin{array}{rl}&\disp\int_{t}^{t+\tau}\int_\Omega   (n_\varepsilon+\varepsilon)  ^{{m}}
\\
=& \int_{t}^{t+\tau}\|  (n_\varepsilon+\varepsilon)  ^{m-1}\|_{L^{\frac{m}({m-1})}(\Omega)}^{\frac{m}({m-1})}
\\
\leq& \gamma_{3}(\int_{t}^{t+\tau}\| \nabla (n_\varepsilon+\varepsilon) ^{m-1}\|_{L^{2}(\Omega)}^{\frac{{m-1}}{m}} \| (n_\varepsilon+\varepsilon) ^{m-1}\|_{L^{\frac{1}{m-1}}(\Omega)}^{\frac{1}{m}}
+\int_{t}^{t+\tau}\| (n_\varepsilon+\varepsilon)  ^{m-1}\|_{L^{\frac{1}{m-1}}(\Omega)})^{\frac{2{m}}{m-1}}\\
\leq& \gamma_{4}\int_{t}^{t+\tau}\| \nabla (n_\varepsilon+\varepsilon)  ^{m-1}\|_{L^{2}(\Omega)}^{2}+\gamma_{4}~~\mbox{for all}~~ t\in(0,T_{max,\varepsilon}-\tau),
\end{array}
\end{equation}
where $\tau=\min\{1,\frac{1}{6}T_{max,\varepsilon}\}.$
Therefore, by \dref{3333cz2.5kke345677ddff89001ddff214114114rrggjjkk}, we conclude that
\begin{equation}
\label{3333cz2.kkk5kke345677ddfdddddf89001ddff214114114rrggjjkk}
\disp\int_{t}^{t+\tau}\int_\Omega   (n_\varepsilon+\varepsilon)  ^{{{m}}}\leq\gamma_{5}~~\mbox{for all}~~ t\in(0,T_{max,\varepsilon}-\tau).
\end{equation}
Thus, for $t\in(0,T_{max,\varepsilon})$, if we write
$$
y(t) :=\|n_\varepsilon(\cdot, t)+\varepsilon\|^{{{m}}}_{L^{{m}}(\Omega)}+
\|\nabla{c_{\varepsilon}}(\cdot, t)\|^{{{2}}}_{L^{{2}}(\Omega)}
$$
and
$$
\rho(t) =2C_{6}^2\int_{\Omega}|\nabla u_{\varepsilon}(\cdot, t)|^2,
$$
\dref{3333cz2.5kksssssss121hhjjj4sddfffdfff114114} implies that
\begin{equation}
y'(t)+h(t)
\leq\disp{ \rho(t)y(t)+C_{11}~\mbox{for all}~t\in(0,T_{max,\varepsilon}),}
\label{ddfghgfhggddhjjjnkkll11cz2.5ghju48}
\end{equation}
where
$$
h(t)=\frac{1}{\gamma_{1}}\frac{{m}({m-1})}{(2{m}-1)^2}\int_\Omega  (n_\varepsilon+\varepsilon)^{2{m}}(\cdot,t)\geq0.
$$
Next, by using estimates \dref{3333cz2.kkk5kke345677ddfdddddf89001ddff214114114rrggjjkk} and  \dref{bnmbncz2.5ghhjuyuivvbnnihjj}, one obtains
$$
\int_{t}^{t+\tau}\rho(s)ds\leq\disp{ C_{12}}
$$
and
$$
\int_{t}^{t+\tau}y(s)ds\leq\disp{ C_{13}},
$$
for all $t\in(0,T_{max,\varepsilon}-\tau)$.
For given $t\in (0, T_{max,\varepsilon})$,  using estimates \dref{3333cz2.kkk5kke345677ddfdddddf89001ddff214114114rrggjjkk} and \dref{bnmbncz2.5ghhjuyuivvbnnihjj} again,
one can choose $t_0 \geq 0$ such that $t_0\in [t-\tau, t)$ and
$$
\disp{ y(\cdot,t_0)\leq C_{14}.}
$$
This, together with \dref{ddfghgfhggddhjjjnkkll11cz2.5ghju48} and the Gronwall lemma, yields
\begin{equation}
\begin{array}{rl}
y(t)\leq&\disp{y(t_0)e^{\int_{t_0}^t\rho(s)ds}+\int_{t_0}^te^{\int_{s}^t\rho(\tau)d\tau}C_{11}ds}
\\
\leq&\disp{C_{14}e^{C_{12}}+\int_{t_0}^te^{C_{12}}C_{11}ds}
\\
\leq&\disp{C_{14}e^{C_{12}}+e^{C_{12}}C_{11}~~\mbox{for all}~~t\in(0,T_{max,\varepsilon}).}
\end{array}
\label{czfvgb2.5ghhddffggjuyghjddddffjjuffghhhddfghhccvjkkklllhhjkkviihjj}
\end{equation}
Finally, collecting \dref{ddfghgfhggddhjjjnkkll11cz2.5ghju48} and \dref{czfvgb2.5ghhddffggjuyghjddddffjjuffghhhddfghhccvjkkklllhhjkkviihjj},
it yields \dref{334444zjscz2.5297x9630222ssdd2114} and \dref{3.10gghhjuuloollsdffffffgghhhy}.
\end{proof}

\begin{lemma}\label{lemma630jklhhjj}
Let $m>1.$ There exists a positive constant $C$ independent of $\varepsilon$, such that
\begin{equation}
\int_{\Omega}{|\nabla u_{\varepsilon}(\cdot,t)|^2}\leq C~~\mbox{for all}~~ t\in(0, T_{max,\varepsilon}).
\label{ddxcvbbggddfgcz2vv.5ghju48cfg924ghyuji}
\end{equation}
\end{lemma}
\begin{proof}
Firstly, applying the Helmholtz projection to both sides of the first equation in \dref{1.1fghyuisda},
then multiplying the result identified by $Au_{\varepsilon}$, integrating by parts, and using the Young inequality, we find that
\begin{equation}
\begin{array}{rl}
&\disp{\frac{1}{{2}}\frac{d}{dt}\|A^{\frac{1}{2}}u_{\varepsilon}\|^{{{2}}}_{L^{{2}}(\Omega)}+
\int_{\Omega}|Au_{\varepsilon}|^2 }
\\
=&\disp{ \int_{\Omega}Au_{\varepsilon}\mathcal{P}(-\kappa
(Y_{\varepsilon}u_{\varepsilon} \cdot \nabla)u_{\varepsilon})+ \int_{\Omega}\mathcal{P}(n_{\varepsilon}\nabla\phi) Au_{\varepsilon}}
\\
\leq&\disp{ \frac{1}{2}\int_{\Omega}|Au_{\varepsilon}|^2+\kappa^2\int_{\Omega}
|(Y_{\varepsilon}u_{\varepsilon} \cdot \nabla)u_{\varepsilon}|^2+ \|\nabla\phi\|^2_{L^\infty(\Omega)}\int_{\Omega}n_{\varepsilon}^2~~\mbox{for all}~~t\in(0,T_{max,\varepsilon})}
\end{array}
\label{ddfghgghjjnnhhkklld911cz2.5ghju48}
\end{equation}
Noticing that $\|Y_{\varepsilon}u_{\varepsilon}\|_{L^2(\Omega)}\leq\|u_{\varepsilon}\|_{L^2(\Omega)},$
it follows from the Gagliardo-Nirenberg inequality and the Cauchy-Schwarz inequality that with some $C_1 >0$ and $C_2 > 0$
\begin{equation}
\begin{array}{rl}
&\kappa^2\disp\int_{\Omega}
|(Y_{\varepsilon}u_{\varepsilon} \cdot \nabla)u_{\varepsilon}|^2\\
\leq&\disp{ \kappa^2\|Y_{\varepsilon}u_{\varepsilon}\|^2_{L^4(\Omega)}\|\nabla u_{\varepsilon}\|^2_{L^4(\Omega)}}
\\
\leq&\disp{ \kappa^2C_1[\|\nabla Y_{\varepsilon}u_{\varepsilon}\|_{L^2(\Omega)}\|Y_{\varepsilon}u_{\varepsilon}\|_{L^2(\Omega)}]
[\|A u_{\varepsilon}\|_{L^2(\Omega)}\|\nabla u_{\varepsilon}\|_{L^2(\Omega)}]
}\\
\leq&\disp{ \kappa^2C_1C_{2}\|\nabla Y_{\varepsilon}u_{\varepsilon}\|_{L^2(\Omega)}
[\|A u_{\varepsilon}\|_{L^2(\Omega)}\|\nabla u_{\varepsilon}\|_{L^2(\Omega)}]
~~\mbox{for all}~~t\in(0,T_{max,\varepsilon})}.
\end{array}
\label{ssdcfvgddfghgghjd911cz2.5ghju48}
\end{equation}
Now, from the fact that $D( A^{\frac{1}{2}})  :=W^{1,2}_0(\Omega;\mathbb{R}^2) \cap L_{\sigma}^{2}(\Omega)$ and
\dref{czfvgb2.5ghhjuyuccvviihjj}, it follows that
\begin{equation}
\|\nabla Y_{\varepsilon}u_{\varepsilon}\|_{L^2(\Omega)}=\|A^{\frac{1}{2}} Y_{\varepsilon}u_{\varepsilon}\|_{L^2(\Omega)}
=\|Y_{\varepsilon} A^{\frac{1}{2}} u_{\varepsilon}\|_{L^2(\Omega)}\leq \|A^{\frac{1}{2}}  u_{\varepsilon}\|_{L^2(\Omega)}\leq \|\nabla  u_{\varepsilon}\|_{L^2(\Omega)}.
\label{ssdcfdhhgghjjnnhhkklld911cz2.5ghju48}
\end{equation}
Due to  Theorem 2.1.1 in \cite{Sohr},  $\|A(\cdot)\|_{L^{2}(\Omega)}$ defines a norm
equivalent to $\|\cdot\|_{W^{2,2}(\Omega)}$ on $D(A)$.
This, together with the Young inequality and estimates \dref{ssdcfdhhgghjjnnhhkklld911cz2.5ghju48} and \dref{ssdcfvgddfghgghjd911cz2.5ghju48}, yields
$$
\begin{array}{rl}
&\kappa^2\disp\int_{\Omega}|(Y_{\varepsilon}u_{\varepsilon} \cdot \nabla)u_{\varepsilon}|^2
\\
\leq&\disp{ C_{3}\|A u_{\varepsilon}\|_{L^2(\Omega)}\|\nabla u_{\varepsilon}\|_{L^2(\Omega)}^2}
\\
\leq&\disp{ \frac{1}{4}\|A u_{\varepsilon}\|_{L^2(\Omega)}+\kappa^4C_{1}^2C_{2}^2\|\nabla u_{\varepsilon}\|_{L^2(\Omega)}^4
~~\mbox{for all}~~t\in(0,T_{max,\varepsilon}),}
\end{array}
$$
which combining with \dref{ddfghgghjjnnhhkklld911cz2.5ghju48} implies that
$$
\disp\frac{1}{{2}}\frac{d}{dt}\|A^{\frac{1}{2}}u_{\varepsilon}\|^{{{2}}}_{L^{{2}}(\Omega)}
\leq\disp{ \kappa^4C_{1}^2C_{2}^2\|\nabla u_{\varepsilon}\|_{L^2(\Omega)}^4+ \|\nabla\phi\|^2_{L^\infty(\Omega)}\int_{\Omega}n_{\varepsilon}^2~\mbox{for all}~t\in(0,T_{max,\varepsilon}),}
$$
By the fact that $\|A^{\frac{1}{2}}u_{\varepsilon}\|^{{{2}}}_{L^{{2}}(\Omega)} = \|\nabla u_{\varepsilon}\|^{{{2}}}_{L^{{2}}(\Omega)},$ we conclude that
\begin{equation}
z'(t)\leq\rho(t)z(t)+ h(t)\disp{~~\mbox{for all}~~t\in(0,T_{max,\varepsilon}),}
\label{ddfghgffgghggddhjjjhjjnnhhkklld911cz2.5ghju48}
\end{equation}
where
$$
z(t) :=\int_{\Omega}|\nabla u_{\varepsilon}(\cdot, t)|^2,
$$
as well as
$$
\rho(t) =2\kappa^4C_{1}^2C_{2}^2\int_{\Omega}|\nabla u_{\varepsilon}(\cdot, t)|^2
$$
and
$$
h(t)=2 \|\nabla\phi\|^2_{L^\infty(\Omega)}\int_{\Omega}n_{\varepsilon}^2(\cdot,t).
$$
However, \dref{bnmbncz2.5ghhjuyuivvbnnihjj} along with \dref{3.10gghhjuuloollsdffffffgghhhy} warrants that for  some  positive constant $\alpha_0$,
\begin{equation}
\label{3333cz2.kkk5kke345677ddfddffddddf89001ddff214114114rrggjjkk}
\disp\int_{t}^{t+\tau}\int_\Omega  |\nabla {u_{\varepsilon}}|^2\leq\alpha_{0}~~\mbox{for all}~~ t\in(0,T_{max,\varepsilon}-\tau)
\end{equation}
and
\begin{equation}
\label{3333cz2.kkk5kke345fffff677ddfdddddf89001ddff214114114rrggjjkk}
\disp\int_{t}^{t+\tau}\int_\Omega n_\varepsilon  ^{2}\leq\alpha_{0}~~\mbox{for all}~~ t\in(0,T_{max,\varepsilon}-\tau)
\end{equation}
with $\tau=\min\{1,\frac{1}{6}T_{max,\varepsilon}\}.$
Now, \dref{3333cz2.kkk5kke345677ddfddffddddf89001ddff214114114rrggjjkk}
and \dref{3333cz2.kkk5kke345fffff677ddfdddddf89001ddff214114114rrggjjkk}    ensure that for all $t\in(0,T_{max,\varepsilon}-\tau)$
$$
\int_{t}^{t+\tau}\rho(s)ds\leq\disp{ 2C_{3}^2\alpha_0}
$$
and
$$
\int_{t}^{t+\tau}h(s)ds\leq\disp{ 4 \|\nabla\phi\|^2_{L^\infty(\Omega)}\alpha_0.}
$$
For given $t\in (0, T_{max,\varepsilon})$, applying \dref{3333cz2.kkk5kke345677ddfddffddddf89001ddff214114114rrggjjkk} again,
we can choose $t_0 \geq 0$ such that $t_0\in [t-\tau, t)$ and
$$
\disp{\int_{\Omega}|\nabla u_{\varepsilon}(\cdot,t_0)|^2\leq C_{4},}
$$
which combined with \dref{ddfghgffgghggddhjjjhjjnnhhkklld911cz2.5ghju48} implies that
\begin{equation}
\begin{array}{rl}
z(t)\leq&\disp{z(t_0)e^{\int_{t_0}^t\rho(s)ds}+\int_{t_0}^te^{\int_{s}^t\rho(\tau)d\tau}h(s)ds}
\\
\leq&\disp{C_{4}e^{2C_{3}^2\alpha_0}+\int_{t_0}^te^{2C_{3}^2\alpha_0}h(s)ds}
\\
\leq&\disp{C_{4}e^{2C_{3}^2\alpha_0}+e^{2C_{3}^2\alpha_0}4 \|\nabla\phi\|^2_{L^\infty(\Omega)}\alpha_{0}~~\mbox{for all}~~t\in(0,T_{max,\varepsilon})}
\end{array}
\label{czfvgb2.5ghhddffggjuyghjjjuffghhhddfghhccvjkkklllhhjkkviihjj}
\end{equation}
by integration. The claimed inequality \dref{ddxcvbbggddfgcz2vv.5ghju48cfg924ghyuji} thus results from \dref{czfvgb2.5ghhddffggjuyghjjjuffghhhddfghhccvjkkklllhhjkkviihjj}.
\end{proof}

\begin{lemma}\label{lemma4563025xxhjkloghyui}
Let $m>1$. Then there exists a positive constant $C$ independent of $\varepsilon$ such that the solution of \dref{1.1fghyuisda} satisfies
\begin{equation}\label{hjui909klopji115}
\disp{\|\nabla c_\varepsilon(\cdot, t)\|_{L^{2m}(\Omega)}\leq C~~\mbox{for all}~~ t\in(0,T_{max}).}
\end{equation}
\end{lemma}

\begin{proof}
Considering the fact that $\nabla c_{\varepsilon}\cdot\nabla\Delta c_{\varepsilon}  = \frac{1}{2}\Delta |\nabla c_{\varepsilon}|^2-|D^2c_{\varepsilon}|^2$,
by a straightforward computation using the second equation in \dref{1.1fghyuisda} and several integrations by parts, we find that
\begin{equation}
\begin{array}{rl}
&\disp{\frac{1}{{2m}}\frac{d}{dt} \|\nabla c_{\varepsilon}\|^{{{2m}}}_{L^{{2m}}(\Omega)}}
\\
= &\disp{\int_{\Omega} |\nabla c_{\varepsilon}|^{2m-2}\nabla c_{\varepsilon}\cdot\nabla(\Delta c_{\varepsilon}
-c_{\varepsilon}+n_{\varepsilon}-u_{\varepsilon}\cdot\nabla  c_{\varepsilon})}
\\
=&\disp{\frac{1}{{2}}\int_{\Omega} |\nabla c_{\varepsilon}|^{2m-2}\Delta |\nabla c_{\varepsilon}|^2
-\int_{\Omega} |\nabla c_{\varepsilon}|^{2m-2}|D^2 c_{\varepsilon}|^2-\int_{\Omega} |\nabla c_{\varepsilon}|^{2m}}
\\
&-\disp{\int_\Omega n_{\varepsilon}\nabla\cdot( |\nabla c_{\varepsilon}|^{2m-2}\nabla c_{\varepsilon})
+\int_\Omega (u_{\varepsilon}\cdot\nabla  c_{\varepsilon})\nabla\cdot( |\nabla c_{\varepsilon}|^{2m-2}\nabla c_{\varepsilon})}
\\
=&\disp{-\frac{\beta-1}{{2}}\int_{\Omega} |\nabla c_{\varepsilon}|^{2m-4}\left|\nabla |\nabla c_{\varepsilon}|^{2}\right|^2
+\frac{1}{{2}}\int_{\partial\Omega} |\nabla c_{\varepsilon}|^{2m-2}\frac{\partial  |\nabla c_{\varepsilon}|^{2}}{\partial\nu}
-\int_{\Omega} |\nabla c_{\varepsilon}|^{2m}}\\
&-\disp{\int_{\Omega} |\nabla c_{\varepsilon}|^{2m-2}|D^2 c_{\varepsilon}|^2
-\int_\Omega n_{\varepsilon} |\nabla c_{\varepsilon}|^{2m-2}\Delta c_{\varepsilon}-\int_\Omega n_{\varepsilon}\nabla c_{\varepsilon}\cdot\nabla( |\nabla c_{\varepsilon}|^{2m-2})}
\\
&+\disp{\int_\Omega (u_{\varepsilon}\cdot\nabla  c_{\varepsilon}) |\nabla c_{\varepsilon}|^{2m-2}\Delta c_{\varepsilon}
+\int_\Omega (u_{\varepsilon}\cdot\nabla  c_{\varepsilon})\nabla c_{\varepsilon}\cdot\nabla( |\nabla c_{\varepsilon}|^{2m-2})}
\\
=&\disp{-\frac{2({m}-1)}{{{m}^2}}\int_{\Omega}\left|\nabla |\nabla c_{\varepsilon}|^{m}\right|^2
+\frac{1}{{2}}\int_{\partial\Omega} |\nabla c_{\varepsilon}|^{2m-2}\frac{\partial  |\nabla c_{\varepsilon}|^{2}}{\partial\nu}
-\int_{\Omega} |\nabla c_{\varepsilon}|^{2m-2}|D^2 c_{\varepsilon}|^2}
\\
&-\disp{\int_\Omega n_{\varepsilon} |\nabla c_{\varepsilon}|^{2m-2}\Delta c_{\varepsilon}
-\int_\Omega n_{\varepsilon}\nabla c_{\varepsilon}\cdot\nabla( |\nabla c_{\varepsilon}|^{2m-2})-\int_{\Omega} |\nabla c_{\varepsilon}|^{2m}}
\\
&+\disp{\int_\Omega (u_{\varepsilon}\cdot\nabla  c_{\varepsilon}) |\nabla c_{\varepsilon}|^{2m-2}\Delta c_{\varepsilon}
+\int_\Omega (u_{\varepsilon}\cdot\nabla  c_{\varepsilon})\nabla c_{\varepsilon}\cdot\nabla( |\nabla c_{\varepsilon}|^{2m-2})}
\end{array}
\label{cz2.5ghju48156}
\end{equation}
for all $t\in(0,T_{max})$.
Here, since $|\Delta c_{\varepsilon}| \leq\sqrt{2}|D^2c_{\varepsilon}|$, by utilizing the Young inequality, we can estimate
\begin{equation}
\begin{array}{rl}
&\disp\int_\Omega n_{\varepsilon} |\nabla c_{\varepsilon}|^{2m-2}\Delta c_{\varepsilon}
\\
\leq&\disp{\sqrt{2}\int_\Omega n_{\varepsilon} |\nabla c_{\varepsilon}|^{2m-2}|D^2c_{\varepsilon}|}
\\
\leq&\disp{\frac{1}{4}\int_\Omega  |\nabla c_{\varepsilon}|^{2m-2}|D^2c_{\varepsilon}|^2+{2}\int_\Omega n^2_{\varepsilon} |\nabla c_{\varepsilon}|^{2m-2}}
\\
\leq&\disp{\frac{1}{4}\int_\Omega  |\nabla c_{\varepsilon}|^{2m-2}|D^2c_{\varepsilon}|^2+{2}\int_\Omega (n_{\varepsilon}+\varepsilon)^2 |\nabla c_{\varepsilon}|^{2m-2}}
\end{array}
\label{cz2.5ghju48hjuikl1}
\end{equation}
and, similarly,
\begin{equation}
\begin{array}{rl}
&\disp\int_\Omega (u_{\varepsilon}\cdot\nabla  c_{\varepsilon}) |\nabla c_{\varepsilon}|^{2{m}-2}\Delta c_{\varepsilon}
\\
\leq&\disp{\sqrt{2}\int_\Omega |u_{\varepsilon}\cdot\nabla  c_{\varepsilon}| |\nabla c_{\varepsilon}|^{2{m}-2}|D^2c_{\varepsilon}|}
\\
\leq&\disp{\frac{1}{4}\int_\Omega  |\nabla c_{\varepsilon}|^{2{m}-2}|D^2c_{\varepsilon}|^2
+2\int_\Omega |u_{\varepsilon}\cdot\nabla  c_{\varepsilon}|^2 |\nabla c_{\varepsilon}|^{2{m}-2}}
\\
\leq&\disp{\frac{1}{4}\int_\Omega  |\nabla c_{\varepsilon}|^{2{m}-2}|D^2c_{\varepsilon}|^2
+2\int_\Omega |u_{\varepsilon}|^2 |\nabla c_{\varepsilon}|^{2{m}}}
\\
\leq&\disp{\frac{1}{4}\int_\Omega  |\nabla c_{\varepsilon}|^{2{m}-2}|D^2c_{\varepsilon}|^2
+2\int_\Omega |u_{\varepsilon}|^2 |\nabla c_{\varepsilon}|^{2{m}}}
\end{array}
\label{cz2.5ghju48hjuikl451}
\end{equation}
for all $t\in(0,T_{max})$. Again, from the Young inequality, we have
\begin{equation}
\begin{array}{rl}
&-\disp\int_\Omega n_{\varepsilon}\nabla c_{\varepsilon}\cdot\nabla( |\nabla c_{\varepsilon}|^{2{m}-2})
\\
= &\disp{-({m}-1)\int_\Omega n_{\varepsilon} |\nabla c_{\varepsilon}|^{2({m}-2)}\nabla c_{\varepsilon}\cdot\nabla |\nabla c_{\varepsilon}|^{2}}
\\
\leq &\disp{\frac{{m}-1}{8}\int_{\Omega} |\nabla c_{\varepsilon}|^{2{m}-4}\left|\nabla |\nabla c_{\varepsilon}|^{2}\right|^2+2({m}-1)
\int_\Omega |n_{\varepsilon}|^2 |\nabla c_{\varepsilon}|^{2{m}-2}}
\\
\leq &\disp{\frac{({m}-1)}{2{{m}^2}}\int_{\Omega}\left|\nabla |\nabla c_{\varepsilon}|^{m}\right|^2+2({m}-1)
\int_\Omega |n_{\varepsilon}|^2 |\nabla c_{\varepsilon}|^{2{m}-2}}
\end{array}
\label{cz2.5ghju4ghjuk81}
\end{equation}
and
\begin{equation}
\begin{array}{rl}
&\int_\Omega (u_{\varepsilon}\cdot\nabla  c_{\varepsilon})\nabla c_{\varepsilon}\cdot\nabla( |\nabla c_{\varepsilon}|^{2{m}-2})
\\
= &\disp{({m}-1)\int_\Omega (u_{\varepsilon}\cdot\nabla  c_{\varepsilon}) |\nabla c_{\varepsilon}|^{2(\beta-2)}\nabla c_{\varepsilon}\cdot
\nabla |\nabla c_{\varepsilon}|^{2}}
\\
\leq &\disp{\frac{{m}-1}{8}\int_{\Omega} |\nabla c_{\varepsilon}|^{2{m}-4}\left|\nabla |\nabla c_{\varepsilon}|^{2}\right|^2}
\\
&+\disp{2({m}-1)\int_\Omega |u_{\varepsilon}\cdot\nabla  c_{\varepsilon}|^2 |\nabla c_{\varepsilon}|^{2{m}-2}}
\\
\leq &\disp{\frac{({m}-1)}{2{{m}^2}}\int_{\Omega}\left|\nabla |\nabla c_{\varepsilon}|^{m}\right|^2
+2({m}-1)\int_\Omega |u_{\varepsilon}|^2 |\nabla c_{\varepsilon}|^{2{m}}.}
\end{array}
\label{cz2.5ghju4ghjuk81}
\end{equation}
Observe that
\begin{equation}
\begin{array}{rl}
&\disp{\int_{\partial\Omega}\frac{\partial |\nabla c_{\varepsilon}|^2}{\partial\nu} |\nabla c_{\varepsilon}|^{2{m}-2} }\\
\leq&\disp{C_\Omega\int_{\partial\Omega} |\nabla c_{\varepsilon}|^{2{m}} }\\
=&\disp{C_\Omega| |\nabla c_{\varepsilon}|^{m}\|^2_{L^2(\partial\Omega)}.}\\
\end{array}
\label{cz2.57151hhkkhhgg}
\end{equation}
Let us take $r\in(0,\frac{1}{2})$. Due to Proposition 4.22 (ii) of \cite{Haroske},
we have that $W^{r+\frac{1}{2},2}(\Omega)\hookrightarrow L^2(\partial\Omega)$ is compact, so that,
\begin{equation}
\begin{array}{rl}
&\disp{\| |\nabla c_{\varepsilon}|^{m}\|^2_{L^2{(\partial\Omega})}\leq C_1\| |\nabla c_{\varepsilon}|^{m}\|^2_{W^{r+\frac{1}{2},2}(\Omega)}.}\\
\end{array}
\label{cz2.57151}
\end{equation}
Now, let us pick $a=\frac{2{m}+2r-1}{2{m}}$. By $r\in(0,\frac{1}{2})$ and $\beta>1$, it implies that $r+\frac{1}{2}\leq a<1$.
Therefore, from the fractional Gagliardo--Nirenberg inequality and Lemma \ref{lemma4556664ddd5630223}, for some positive constants $\delta_0,\delta_1$ and $C_1$,
we conclude
\begin{equation}
\begin{array}{rl}
&\disp{\| |\nabla c_{\varepsilon}|^{m}\|^2_{W^{r+\frac{1}{2},2}(\Omega)}}
\\
\leq&\disp{\delta_0\|\nabla |\nabla c_{\varepsilon}|^{m}\|^a_{L^2(\Omega)}\| |\nabla c_{\varepsilon}|^\beta\|^{1-a}_{L^{\frac{2}{m}}(\Omega)}
+\delta_1\| |\nabla c_{\varepsilon}|^\beta\|_{L^{\frac{2}{m}}(\Omega)}}
\\
\leq&\disp{C_1\|\nabla |\nabla c_{\varepsilon}|^{m}\|^a_{L^2(\Omega)}+C_1}.
\end{array}
\label{vvggcz2.57151}
\end{equation}
Combining \dref{cz2.57151hhkkhhgg}--\dref{vvggcz2.57151}, using the Young inequality and the fact that $a\in(0,1)$, it yields
\begin{equation}
\begin{array}{rl}
&\disp\int_{\partial\Omega}\frac{\partial |\nabla c_{\varepsilon}|^2}{\partial\nu} |\nabla c_{\varepsilon}|^{2{m}-2}
\\
\leq &\disp{C_2\|\nabla |\nabla c_{\varepsilon}|^{m}\|^a_{L^2(\Omega)}+C_2}
\\
\leq &\disp{\frac{({m}-1)}{2{{m}^2}}\int_{\Omega}\left|\nabla |\nabla c_{\varepsilon}|^{m}\right|^2+C_3}.
\end{array}
\label{cz2.57151hhkkhhggyyxx}
\end{equation}
Now, together with \dref{cz2.5ghju48156}--\dref{cz2.5ghju4ghjuk81} and \dref{cz2.57151hhkkhhggyyxx}, we can derive that, for some positive constant $C_4$,
\begin{equation}\label{hjui909klopsssdddji115}
\begin{array}{rl}
&\disp{\frac{1}{{2{m}}}\frac{d}{dt}\|\nabla c_{\varepsilon}\|^{{{2{m}}}}_{L^{{2{m}}}(\Omega)}
+\frac{{m}-1}{2{{m}^2}}\int_{\Omega}\left|\nabla |\nabla c_{\varepsilon}|^{{m}}\right|^2
+\frac{1}{2}\int_\Omega  |\nabla c_{\varepsilon}|^{2{m}-2}|D^2c_{\varepsilon}|^2+\int_{\Omega} |\nabla c_{\varepsilon}|^{2{m}}}
\\
\leq&\disp{2{m}\int_\Omega n^2_{\varepsilon} |\nabla c_{\varepsilon}|^{2{m}-2}
+2{m}\int_\Omega |u_{\varepsilon}|^2 |\nabla c_{\varepsilon}|^{2{m}}+C_4~~\mbox{for all}~~ t\in(0,T_{max}).}
\end{array}
\end{equation}
We proceed to estimate the first term on the right-hand side of \dref{hjui909klopsssdddji115}. By using the Young inequality, we conclude that
\begin{equation}\label{hjui909klopsddddssdddji115}
\begin{array}{rl}
&\disp{2{m}\int_\Omega n^2_{\varepsilon} |\nabla c_{\varepsilon}|^{2{m}-2}}
\\
\leq&\disp{2{m}\int_\Omega (n_{\varepsilon}+\varepsilon)^2 |\nabla c_{\varepsilon}|^{2{m}-2}}
\\
\leq&\disp{\frac{1}{2}\int_\Omega|\nabla c_{\varepsilon}|^{2{m}}+C_5\int_\Omega (n_{\varepsilon}+\varepsilon)^{2m}~~\mbox{for all}~~ t\in(0,T_{max})}
\end{array}
\end{equation}
and
\begin{equation}\label{hjui909klopsddddssdddjssssi115}
2{m}\disp\int_\Omega |u_{\varepsilon}|^2 |\nabla c_{\varepsilon}|^{2{m}}
\leq\disp{\int_\Omega  |\nabla c_{\varepsilon}|^{2{m}+1}+C_6\int_\Omega  u_{\varepsilon}^{4{m}+2}~~\mbox{for all}~~ t\in(0,T_{max}),}
\end{equation}
where $C_5=\frac{m}{m-1}\left(\frac{1}{2}m\right)^{-\frac{1}{m-1}}(2{m})^{m}$ and $C_6=(2{m})^{2{m}+1}$.
On the other hand, due to \dref{334444zjscz2.5297x9630222ssdd2114}, we derive from the Gagliardo--Nirenberg inequality that for some positive constants $C_7$ and $C_8$
$$
\begin{array}{rl}
&\disp{\int_\Omega  |\nabla c_{\varepsilon}|^{2{m}+1}}
\\
=&\disp{\| |\nabla c_{\varepsilon}|^{m}\|_{L^{\frac{2{m}+1}{m}}(\Omega)}^{\frac{2{m}+1}{m}}}
\\
\leq&\disp{C_7(\|\nabla |\nabla c_{\varepsilon}|^m\|_{L^2(\Omega)}^{\frac{2m-1}{2m+1}}\| |\nabla c_{\varepsilon}|^m\|_{L^\frac{2}{m}(\Omega)}^{\frac{2}{2m+1}}
+\|   |\nabla c_{\varepsilon}|^m\|_{L^\frac{2}{m}
(\Omega)})^{\frac{2{m}+1}{m}}}
\\
\leq&\disp{C_8(\|\nabla |\nabla c_{\varepsilon}|^m\|_{L^2(\Omega)}^{\frac{2m-1}{m}}+1),}
\end{array}
$$
which together with the Young inequality provides a constant $C_9$  such that
\begin{equation}
\disp{\int_\Omega  |\nabla c_{\varepsilon}|^{2{m}+1}}\leq\disp{\frac{{m}-1}{2{{m}^2}}\int_{\Omega}
\left|\nabla |\nabla c_{\varepsilon}|^{{m}}\right|^2+C_9.}
\label{cz2.563022222ikosssplsss255}
\end{equation}
Inserting \dref{cz2.563022222ikosssplsss255} into \dref{hjui909klopsddddssdddjssssi115}, we derive that
\begin{equation}\label{hddssddjui909klopsddddssdddjssssi115}
\begin{array}{rl}
&2{m}\disp\int_\Omega |u_{\varepsilon}|^2 |\nabla c_{\varepsilon}|^{2{m}}
\\
\leq&\disp{\frac{{m}-1}{2{{m}^2}}\int_{\Omega}
\left|\nabla |\nabla c_{\varepsilon}|^{{m}}\right|^2+C_6\int_\Omega  u_{\varepsilon}^{4{m}+2}+C_9~\mbox{for all}~ t\in(0,T_{max}).}
\end{array}
\end{equation}
Substituting \dref{hjui909klopsddddssdddji115} and \dref{hddssddjui909klopsddddssdddjssssi115}  into \dref{hjui909klopsssdddji115}, we have
$$
\begin{array}{rl}
&\disp\frac{1}{{2{m}}}\frac{d}{dt}\|\nabla c_{\varepsilon}\|^{{{2{m}}}}_{L^{{2{m}}}(\Omega)}+\frac{1}{{2{}}}\int_{\Omega} |\nabla c_{\varepsilon}|^{2{m}}
\\
\leq&\disp{C_5\int_\Omega (n_{\varepsilon}+\varepsilon)^{2m}+C_6\int_\Omega  u_{\varepsilon}^{4{m}+2}+C_{10}~\mbox{for all}~ t\in(0,T_{max}).}
\end{array}
$$
Next, since $W^{1,2}(\Omega)\hookrightarrow L^p(\Omega)$ for any $p>1,$ the boundedness of $\|\nabla u_{\varepsilon}(\cdot, t)\|_{L^2(\Omega)}$
(see Lemma \ref{lemma630jklhhjj}) implies that there exists a positive constant $C_{11}$ such that
$$
\|u_\varepsilon(\cdot, t)\|_{L^{4m+2}(\Omega)}\leq  C_{11}~~ \mbox{for all}~~ t\in(0,T_{max,\varepsilon}),
$$
which together with \dref{3.10gghhjuuloollsdffffffgghhhy} yields to \dref{hjui909klopji115} by using Lemma \ref{lemma630}.
This completes the proof of Lemma \ref{lemma4563025xxhjkloghyui}.
\end{proof}

\begin{lemma}\label{lemma456ddd3025xxhjjjjjkloghyui}
Let $m>1$. Then for all $p>1,$ there exists a positive constant $C$ independent of $\varepsilon$, such that
the solution of \dref{1.1fghyuisda} from Lemma \ref{lemma70} satisfies
\begin{equation}\label{hjuddddidddddddd909klopddffji115}
\begin{array}{rl}
&\disp{\|n_\varepsilon(\cdot, t)\|_{L^{p}(\Omega)}\leq C~~\mbox{for all}~~ t\in(0,T_{max}).}\\
\end{array}
\end{equation}
\end{lemma}

\begin{proof}
Let $p>\max\{1,m-1\}$.
Taking ${(n_{\varepsilon}+\varepsilon)^{p-1}}$ as the test function for the first equation of $\dref{1.1fghyuisda}$,
combining with the second equation, and using \dref{x1.73142vghf48gg}, the Young inequality and the fact $\nabla\cdot u_\varepsilon=0$,
we obtain, for all $t\in(0,T_{max,\varepsilon})$,
$$
\begin{array}{rl}
&\disp{\frac{1}{{p}}\frac{d}{dt}\|n_{\varepsilon}+\varepsilon\|^{{{p}}}_{L^{{p}}(\Omega)}+
m(p-1)\int_{\Omega} (n_{\varepsilon}+\varepsilon)^{m+p-3}|{\nabla} {n}_{\varepsilon}|^2 }
\\
\leq&\disp{(p-1)\int_\Omega  (n_{\varepsilon}+\varepsilon)^{p-2}n_{\varepsilon}{|\nabla} {n}_{\varepsilon}||S_{\varepsilon}(x,n_{\varepsilon},c_{\varepsilon}) ||\nabla c_{\varepsilon}| }
\\
\leq&\disp{(p-1)C_S\int_\Omega  (n_{\varepsilon}+\varepsilon)^{p-1}{|\nabla} {n}_{\varepsilon}||\nabla c_{\varepsilon}| }
\\
\leq&\disp{\frac{m(p-1)}{2}\int_{\Omega}(n_{\varepsilon}+\varepsilon)^{m+p-3} |\nabla n_{\varepsilon}|^2
+\frac{(p-1)C_S^2}{2m}\int_\Omega{(n_{\varepsilon}+\varepsilon)^{p+1-m}}|\nabla c_{\varepsilon}|^2},
\end{array}
$$
which implies  that
\begin{equation}
\begin{array}{rl}
&\disp\frac{1}{{p}}\disp\frac{d}{dt}\|n_{\varepsilon}+\varepsilon\|^{{{p}}}_{L^{{p}}(\Omega)}+
\disp\frac{m(p-1)}{2}\int_{\Omega}(n_{\varepsilon}+\varepsilon)^{m+p-3} |\nabla n_{\varepsilon}|^2
\\
\leq&\disp{\frac{(p-1)C_S^2}{2m}\int_\Omega (n_{\varepsilon}+\varepsilon)^{p+1-m}|\nabla c_{\varepsilon}|^2}
\end{array}
\label{cz2.5ghhjuyuiihjj}
\end{equation}
for all $t\in(0,T_{max,\varepsilon})$. In the following, we will estimate the right-hand side of \dref{cz2.5ghhjuyuiihjj}.
In fact, due to $m>1$, we conclude from \dref{hjui909klopji115} that
$$
\begin{array}{rl}
&\disp{ \int_\Omega (n_{\varepsilon}+\varepsilon)^{p+1-m} |\nabla c_{\varepsilon}|^2}
\\
\leq&\disp{ \left(\int_\Omega(n_{\varepsilon}
+\varepsilon)^{\frac{m(p+1-m)}{m-1}}\right)^{\frac{m-1}{m}}\left(\int_\Omega |\nabla c_{\varepsilon}|^{2m}\right)^{\frac{1}{m}}}
\\
\leq&\disp{ C_1\left(\int_\Omega(n_{\varepsilon}
+\varepsilon)^{\frac{m(p+1-m)}{m-1}}\right)^{\frac{m-1}{m}}~~\mbox{for all}~~ t\in(0,T_{max})}
\end{array}
$$
by using the  H\"{o}lder inequality.
These  together with \dref{czfvgb2.5ghhjuyuccvviihjj} and $m>1$ implies that
\begin{equation}
\begin{array}{rl}
&\disp{C_1\left(\int_\Omega(n_{\varepsilon}+\varepsilon)^{\frac{m(p+1-m)}{m-1}}\right)^{\frac{m-1}{m}}}
\\
=&\disp{C_1\|(n_{\varepsilon}+\varepsilon)^{\frac{m(p+1-m)}{m-1}}\|^{\frac{2(p+1-m)}{m+p-1}}_{L^{\frac{2m(p+1-m)}{(m-1)(m+p-1)}}(\Omega)}}
\\
\leq&C_2(\|\nabla (n_{\varepsilon}+\varepsilon)^{\frac{p+m-1}{2}}\|_{L^2(\Omega)}^{\frac{mp-m^2+1}{m(p+1-m)}}
\|(n_{\varepsilon}+\varepsilon)^{\frac{p+m-1}{2}}\|_{L^\frac{2}{p+m-1}(\Omega)}^{\frac{m-1}{m(p+1-m)}}
\\
&+\|(n_{\varepsilon}+\varepsilon)^{\frac{p+m-1}{2}}\|_{L^\frac{2}{p+m-1}(\Omega)})^{\frac{2(p+1-m)}{m+p-1}}
\\
\leq&\disp{C_{3}(\|\nabla   (n_{\varepsilon}+\varepsilon)^{\frac{p+m-1}{2}}\|_{L^2(\Omega)}^{\frac{2(mp-m^2+1)}{m(p+m-1)}}+1)}
\\
\leq&\disp{\frac{m(p-1)}{4}\int_{\Omega}(n_{\varepsilon}+\varepsilon)^{m+p-3} |\nabla n_{\varepsilon}|^2+C_4~~\mbox{for all}~~ t\in(0,T_{max})}
\end{array}
\label{cz2.dddd563022222ikopl2sdfg44}
\end{equation}
by using the Gagliardo--Nirenberg inequality as well as the Young inequality and the fact that
$$\frac{2(mp-m^2+1)}{m(p+m-1)}<2.$$
Inserting \dref{cz2.dddd563022222ikopl2sdfg44} into \dref{cz2.5ghhjuyuiihjj}, we have
$$
\disp\frac{1}{{p}}\disp\frac{d}{dt}\|n_{\varepsilon}+\varepsilon\|^{{{p}}}_{L^{{p}}(\Omega)}+
\disp\frac{m(p-1)}{4}\int_{\Omega}(n_{\varepsilon}+\varepsilon)^{m+p-3} |\nabla n_{\varepsilon}|^2\leq\disp{C_5.}
$$
Therefore, \dref{hjuddddidddddddd909klopddffji115} holds by using Lemma \ref{lemma630} and some basic calculation.
This completes the proof of Lemma \ref{lemma456ddd3025xxhjjjjjkloghyui}.
\end{proof}

\begin{lemma}\label{lemma45630hhuujj}
Let $m>1$ and $\gamma\in(\frac{1}{2},1).$  Then one can find a positive constant $C$ independent of $\varepsilon$, such that
$$
\|n_\varepsilon(\cdot,t)\|_{L^\infty(\Omega)}  \leq C ~~\mbox{for all}~~ t\in(0,T_{max,\varepsilon}),
$$
$$
\|c_\varepsilon(\cdot,t)\|_{W^{1,\infty}(\Omega)}  \leq C ~~\mbox{for all}~~ t\in(0,T_{max,\varepsilon})
$$
as well as
$$
\|u_\varepsilon(\cdot,t)\|_{L^{\infty}(\Omega)}  \leq C ~~\mbox{for all}~~ t\in(0,T_{max,\varepsilon})
$$
and
$$
\|A^\gamma u_\varepsilon(\cdot,t)\|_{L^{2}(\Omega)}  \leq C ~~\mbox{for all}~~ t\in(0,T_{max,\varepsilon}).
$$
\end{lemma}

\begin{proof}
Firstly, applying the variation-of-constants formula to the projected version of the third
equation in \dref{1.1fghyuisda}, we derive that
$$
u_\varepsilon(\cdot, t) = e^{-tA}u_0 +\int_0^te^{-(t-\tau)A}
\mathcal{P}[n_\varepsilon(\cdot,t)\nabla\phi-\kappa
(Y_{\varepsilon}u_{\varepsilon} \cdot \nabla)u_{\varepsilon}]d\tau~~ \mbox{for all}~~ t\in(0,T_{max,\varepsilon}).
$$
Now, picking $h_{\varepsilon}=\mathcal{P}[n_\varepsilon(\cdot,t)\nabla\phi-\kappa
(Y_{\varepsilon}u_{\varepsilon} \cdot \nabla)u_{\varepsilon}]$,
then, in view  of the  standard smoothing properties of the Stokes semigroup, we derive that for all $t\in(0,T_{max,\varepsilon})$ and $\gamma\in ( \frac{1}{2}, 1)$,
there exist $C_{1} > 0$ and $C_{2} > 0$ such that
\begin{equation}
\begin{array}{rl}
&\|A^\gamma u_{\varepsilon}(\cdot, t)\|_{L^2(\Omega)}
\\
\leq&\disp{\|A^\gamma e^{-tA}u_0\|_{L^2(\Omega)}
+\int_0^t\|A^\gamma e^{-(t-\tau)A}h_{\varepsilon}(\cdot,\tau)d\tau\|_{L^2(\Omega)}d\tau}
\\
\leq&\disp{\|A^\gamma u_0\|_{L^2(\Omega)}
+C_{1}\int_0^t(t-\tau)^{-\gamma-\frac{2}{2}(\frac{1}{p_0}-\frac{1}{2})}e^{-\lambda(t-\tau)}\|h_{\varepsilon}(\cdot,\tau)\|_{L^{p_0}(\Omega)}d\tau}
\\
\leq&\disp{C_{2}+C_{1}\int_0^t(t-\tau)^{-\gamma-\frac{2}{2}(\frac{1}{p_0}-\frac{1}{2})}e^{-\lambda(t-\tau)}\|h_{\varepsilon}(\cdot,\tau)\|_{L^{p_0}(\Omega)}d\tau}
\end{array}
\label{ssddcz2.571hhhhh51222ccvvhddfccvvhjjjkkhhggjjllll}
\end{equation}
by using \dref{ccvvx1.731426677gg}, where $p_0\in (1,2)$ satisfies that
\begin{equation}
p_0>\frac{2}{3-2\gamma}.
\label{cz2.571hhhhh51222ccvvhddfccffgghhhhvvhjjjkkhhggjjllll}
\end{equation}
In light of \dref{hjuddddidddddddd909klopddffji115}, for some positive constant $C_3$, it has
$$
\disp{\|n_\varepsilon(\cdot, t)\|_{L^{p_0}(\Omega)}\leq C_3~~\mbox{for all}~~ t\in(0,T_{max})}.
$$
Employing the H\"{o}lder inequality and the continuity of $\mathcal{P}$ in $L^p(\Omega;\mathbb{R}^2)$ (see \cite{Fujiwara66612186}),
there exist positive constants $C_{4},C_{5}, C_{6}$ and  $C_{7}$ such that
\begin{equation}
\begin{array}{rl}
&\|h_{\varepsilon}(\cdot,t)\|_{L^{p_0}(\Omega)}
\\
\leq& C_{4}\|(Y_{\varepsilon}u_{\varepsilon} \cdot \nabla)u_{\varepsilon}(\cdot,t)\|_{L^{p_0}(\Omega)}+ C_{4}\|n_\varepsilon(\cdot,t)\|_{L^{p_0}(\Omega)}
\\
\leq& C_{5}\|Y_{\varepsilon}u_{\varepsilon}\|_{L^{\frac{2p_0}{2-p_0}}(\Omega)} \|\nabla u_{\varepsilon}(\cdot,t)\|_{L^{2}(\Omega)}+ C_{5}
\\
\leq& C_{6}\|\nabla Y_{\varepsilon}u_{\varepsilon}\|_{L^{2}(\Omega)} \|\nabla u_{\varepsilon}(\cdot,t)\|_{L^{2}(\Omega)}+ C_{5}
\\
\leq& C_{7}~~~\mbox{for all}~~ t\in(0,T_{max,\varepsilon}),
\end{array}
\label{cz2.571hhhhh5122ddfddfffgg2ccvvhddfccffgghhhhvvhjjjkkhhggjjllll}
\end{equation}
where we have used the fact that $W^{1,2}(\Omega)\hookrightarrow L^\frac{2p_0}{2-p_0}(\Omega)$ and the boundedness of $\|\nabla u_{\varepsilon}(\cdot,t)\|_{L^{2}(\Omega)}.$
Collecting  \dref{ssddcz2.571hhhhh51222ccvvhddfccvvhjjjkkhhggjjllll},
\dref{cz2.571hhhhh51222ccvvhddfccffgghhhhvvhjjjkkhhggjjllll} and \dref{cz2.571hhhhh5122ddfddfffgg2ccvvhddfccffgghhhhvvhjjjkkhhggjjllll},
we conclude that
$$
\begin{array}{rl}
&\|A^\gamma u_{\varepsilon}(\cdot, t)\|_{L^2(\Omega)}
\\
\leq&\disp{C_{8}\int_0^t(t-\tau)^{-\gamma-\frac{2}{2}(\frac{1}{p_0}-\frac{1}{2})}e^{-\lambda(t-\tau)}\|h_{\varepsilon}(\cdot,\tau)\|_{L^{p_0}(\Omega)}d\tau}
\\
\leq&\disp{C_{9}\int_0^t(t-\tau)^{-\gamma-\frac{2}{2}(\frac{1}{p_0}-\frac{1}{2})}e^{-\lambda(t-\tau)}\|h_{\varepsilon}(\cdot,\tau)\|_{L^{p_0}(\Omega)}d\tau
~~\mbox{for all}~~ t\in(0,T_{max,\varepsilon}),}
\end{array}
$$
which together with the fact that $D(A^\gamma)$ is continuously embedded into $L^\infty(\Omega)$ by $\gamma>\frac{1}{2}$ yields
\begin{equation}
\|u_{\varepsilon}(\cdot, t)\|_{L^\infty(\Omega)}\leq  C_{10}~~ \mbox{for all}~~ t\in(0,T_{max,\varepsilon}).
\label{cz2.5jkkcvvvhjdsdfffffdkfffffkhhgll}
\end{equation}
In view of \dref{cz2.5jkkcvvvhjdsdfffffdkfffffkhhgll} and \dref{hjui909klopji115}, we may use \dref{ccvvx1.731426677gg},
the fact that $m>1$, and the smoothing properties of the Neumann heat semigroup $(e^{t\Delta})_{t\geq0}$ to see that there exists $C_{11} > 0$ such that
\begin{equation}
\|\nabla c_{\varepsilon}(\cdot, t)\|_{L^{{\infty}}(\Omega)}\leq C_{11} ~~ \mbox{for all}~~~  t\in(0,T_{max}).
\label{cz2.5g5gghh56ssss789hhjui78jj90099}
\end{equation}
Then, the boundedness of $n_{\varepsilon}$ can be obtained by the well-known Moser-Alikakos iteration procedure (see  e.g. Lemma A.1 of  \cite{Tao794}).
Indeed, by using \dref{cz2.5jkkcvvvhjdsdfffffdkfffffkhhgll}  and \dref{cz2.5g5gghh56ssss789hhjui78jj90099}, we see that the hypotheses of Lemma A.1 of \cite{Tao794} are valid provided
that we take the parameter $p$ in Lemma  \ref{lemma456ddd3025xxhjjjjjkloghyui} appropriately large. Thus, we obtain
$$
\|n_{\varepsilon}(\cdot, t)\|_{L^{{\infty}}(\Omega)}\leq C_{12} ~~ \mbox{for all}~~~  t\in(0,T_{max}).
$$
The proof of Lemma \ref{lemma45630hhuujj} is completed.
\end{proof}

With all above regularization properties of each component $n_{\varepsilon}$, $c_{\varepsilon}$, $u_{\varepsilon}$ at hand, we can show the
existence of global bounded  solutions to the regularized system \dref{1.1fghyuisda}.

\begin{lemma}\label{lemma45630hhuujjuu}
Let $m> 1$ and  $\gamma\in(\frac{1}{2},1).$.
Let $(n_\varepsilon, c_\varepsilon, u_\varepsilon, P_\varepsilon)_{\varepsilon\in(0,1)}$ be classical solutions of \dref{1.1fghyuisda} constructed in Lemma \ref{lemma70} on $[0, T_{max})$.
Then the solution is global on $[0,\infty)$. Moreover, one can find $C > 0$ independent of $\varepsilon\in(0, 1)$ such that
$$
\|n_\varepsilon(\cdot,t)\|_{L^\infty(\Omega)}  \leq C ~~\mbox{for all}~~ t\in(0,\infty)
$$
and
$$
\|c_\varepsilon(\cdot,t)\|_{W^{1,\infty}(\Omega)}  \leq C ~~\mbox{for all}~~ t\in(0,\infty)
$$
as well as
$$
\|u_\varepsilon(\cdot,t)\|_{W^{1,\infty}(\Omega)}  \leq C ~~\mbox{for all}~~ t\in(0,\infty).
$$
In addition, we also have
$$
\|A^\gamma u_\varepsilon(\cdot,t)\|_{L^{2}(\Omega)}  \leq C ~~\mbox{for all}~~ t\in(0,\infty).
$$
\end{lemma}

Then, with the help of Lemma \ref{lemma45630hhuujjuu},
we can straightforwardly deduce the uniform H\"{o}lder properties of $c_\varepsilon,\nabla c_\varepsilon$ and $u_\varepsilon$
by the standard parabolic regularity theory as the proof of Lemmas 3.18--3.19 in \cite{Winkler11215} (see also \cite{Zhengsdsd6}).

\begin{lemma}\label{lemma45630hhuujjuuyy}
Let $m> 1$. Then one can find $\mu\in(0, 1)$ such that for some $C > 0$
$$
\|c_\varepsilon(\cdot,t)\|_{C^{\mu,\frac{\mu}{2}}(\Omega\times[t,t+1])}  \leq C ~~\mbox{for all}~~ t\in(0,\infty)
$$
as well as
$$
\|u_\varepsilon(\cdot,t)\|_{C^{\mu,\frac{\mu}{2}}(\Omega\times[t,t+1])} \leq C ~~\mbox{for all}~~ t\in(0,\infty),
$$
and for any $\tau> 0$ there exists $C(\tau) > 0$ fulfilling
$$
\|\nabla c_\varepsilon(\cdot,t)\|_{C^{\mu,\frac{\mu}{2}}(\Omega\times[t,t+1])} \leq C ~~\mbox{for all}~~ t\in(\tau,\infty).
$$
\end{lemma}

\section{Prove of the main result}

In this section, we will give the prove of the main result.
Based on the above lemmas, we will construct a weak solution as the limit of classical solutions to approximating systems \dref{1.1fghyuisda}.
Applying the idea of \cite{Zhengsdsd6} (see also \cite{Winkler11215} and \cite{Liuddfffff}), we first state the definition of the solution as follows.

\begin{definition}\label{df1}
Let $T > 0$ and $(n_0, c_0, u_0)$ fulfills \dref{ccvvx1.731426677gg}.
Then a triple of functions $(n, c, u)$ is called a weak solution of (\ref{0.4})-(\ref{0.5}) if the following conditions are satisfied
$$
\left\{\begin{array}{ll}
n\in L_{loc}^1(\bar{\Omega}\times[0,T)),
\\
c \in L_{loc}^1([0,T); W^{1,1}(\Omega)),
\\
u \in  L_{loc}^1([0,T); W^{1,1}(\Omega)),
\end{array}\right.
$$
where $n\geq 0$ and $c\geq 0$ in $\Omega\times(0, T)$ as well as $\nabla\cdot u = 0$ in the distributional sense in $\Omega\times(0, T)$, moreover,
$$
\begin{array}{rl}
&~~ n^m~\mbox{belong to}~~ L^1_{loc}(\bar{\Omega}\times [0, \infty)),\\
&cu,~ ~nu ~~\mbox{and}~~n\nabla c~ \mbox{belong to}~~
L^1_{loc}(\bar{\Omega}\times [0, \infty);\mathbb{R}^{2})
\end{array}
$$
and
$$
\disp{-\int_0^{T}\int_{\Omega}n\varphi_t-\int_{\Omega}n_0\varphi(\cdot,0)}
=\disp{\int_0^T\int_{\Omega}n^m\Delta\varphi+\int_0^T\int_{\Omega}n\nabla c\cdot\nabla\varphi}
+\disp{\int_0^T\int_{\Omega}nu\cdot\nabla\varphi}
$$
for any $\varphi\in C_0^{\infty} (\bar{\Omega}\times[0, T))$ satisfying $\frac{\partial\varphi}{\partial\nu}= 0$ on $\partial\Omega\times (0, T)$, as well as
$$
\begin{array}{rl}
&\disp{-\int_0^{T}\int_{\Omega}c\varphi_t-\int_{\Omega}c_0\varphi(\cdot,0)}
\\
=&\disp{-\int_0^T\int_{\Omega}\nabla c\cdot\nabla\varphi-\int_0^T\int_{\Omega}c\varphi+\int_0^T\int_{\Omega}n\varphi+\int_0^T\int_{\Omega}cu\cdot\nabla\varphi}
\end{array}
$$
for any $\varphi\in C_0^{\infty} (\bar{\Omega}\times[0, T))$  and
$$
\begin{array}{rl}
&\disp{-\int_0^{T}\int_{\Omega}u\varphi_t-\int_{\Omega}u_0\varphi(\cdot,0) }
\\
=&\disp{\kappa\int_0^T\int_{\Omega} u\otimes u\cdot\nabla\varphi-\int_0^T\int_{\Omega}\nabla u\cdot\nabla\varphi-
\int_0^T\int_{\Omega}n\nabla\phi\cdot\varphi}
\end{array}
$$
for any $\varphi\in C_0^{\infty} (\bar{\Omega}\times[0, T);\mathbb{R}^2)$ fulfilling $\nabla\varphi\equiv 0$ in $\Omega\times(0, T)$.
If for each $T>0$, $(n, c, u)$ :$\Omega\times (0,\infty)\longrightarrow \mathbb{R}^4$ is a weak solution of (\ref{0.4})-(\ref{0.5}) in $\Omega\times(0, T)$,
then we call $(n, c, u)$ a global weak solution of (\ref{0.4})-(\ref{0.5}).
\end{definition}

In order to use the Aubin-Lions Lemma (see e.g. \cite{Simon}), we will need the regularity of the time derivative of bounded solutions.
Employing almost exactly the same arguments as that in the proof of Lemmas 3.22--3.23 in \cite{Winkler11215} (the minor necessary changes are left as an easy exercise to the reader),
and taking advantage of Lemma \ref{lemma45630hhuujjuu}, we conclude the following Lemma.

\begin{lemma}\label{lemma45630hhuujjuuyytt}
Let $m> 1$ and $\varsigma> \max\{m,2(m - 1 )\}$. Then for all $\varepsilon\in(0, 1)$, there exists a positive constant $C$ independent of $\varepsilon$ such that
$$
\|\partial_tn_\varepsilon(\cdot,t)\|_{(W^{2,2}_0(\Omega))^*}  \leq C ~~\mbox{for all}~~ t\in(0,\infty).
$$
Moreover, let $\varsigma> \max\{m,2(m - 1 )\}$. Then for all $T > 0$ and $\varepsilon\in(0,1)$, one can find $C(T)$  independent of $\varepsilon$ such that
$$
\int_0^T\|\partial_t(n_{\varepsilon}+\varepsilon)^\varsigma(\cdot,t)\|_{(W^{2,2}_0(\Omega))^*}dt  \leq C(T) ~~\mbox{for all}~~ t\in(0,T)
$$
and
$$
\int_{0}^T\int_{\Omega} |\nabla (n_{\varepsilon}+\varepsilon)^{\varsigma}|^2\leq C(T)~~\mbox{for all}~~ t\in(0,T).
$$
\end{lemma}

Finally, we can prove the main result.

{\bf Proof of Theorem \ref{theorem3}}.
In conjunction with Lemma \ref{lemma45630hhuujjuu} and the Aubin-Lions compactness lemma (see e.g. Simon \cite{Simon}),
we thus infer the existence of a sequence of numbers $\varepsilon = \varepsilon_j \searrow 0$ along which
\begin{equation}
 n_\varepsilon\rightharpoonup n ~~\mbox{weakly star in}~~ L^\infty(\Omega\times(0,\infty)),
 \label{zjscz2.5297x9630222222ee}
\end{equation}
\begin{equation}
n_\varepsilon\rightarrow n ~~\mbox{in}~~ C^0_{loc}([0,\infty); (W^{2,2}_0 (\Omega))^*),
\label{zjscz2.5297x96302222tt}
\end{equation}
\begin{equation}
c_\varepsilon\rightarrow c ~~\mbox{in}~~ C^0_{loc}(\bar{\Omega}\times[0,\infty)),
 \label{zjscz2.5297x96302222tt3}
\end{equation}
\begin{equation}
\nabla c_\varepsilon\rightarrow \nabla c ~~\mbox{in}~~ C^0_{loc}(\bar{\Omega}\times(0,\infty)),
 \label{zjscz2.5297x96302222tt4}
\end{equation}
\begin{equation}
\nabla c_\varepsilon\rightharpoonup \nabla c ~~\mbox{weakly star in}~~ L^{\infty}(\Omega\times(0,\infty))
 \label{zjscz2.5297x96302222tt4}
\end{equation}
as well as
\begin{equation}
u_\varepsilon\rightarrow u ~~\mbox{in}~~ C^0_{loc}(\bar{\Omega}\times[0,\infty))
\label{zjscz2.5297x96302222tt44}
\end{equation}
and
\begin{equation}
D u_\varepsilon\rightharpoonup Du ~~\mbox{weakly star in}~~L^{\infty}(\Omega\times(0,\infty))
 \label{zjscz2.5297x96302222tt4455}
\end{equation}
holds for some limit $(n,c,u) \in (L^\infty(\Omega\times  (0,\infty)))^4$ with nonnegative $n $ and $c$.
On the other hand, Lemma \ref{lemma45630hhuujjuuyytt} implies that for each $T > 0,$ $(n_{\varepsilon}^\varsigma)_{\varepsilon\in(0,1)}$ is bounded in $L^2((0, T);W^{1,2}(\Omega))$,
so that, using Aubin-Lions lemma again,
one may obtain $n_{\varepsilon}^\varsigma\rightarrow z^\varsigma$ for some nonnegative measurable $z:\Omega\times(0,\Omega)\rightarrow\mathbb{R}$.
Thus, \dref{zjscz2.5297x9630222222ee} and the Egorov theorem yields to  $z=n$ necessarily, and thereby
\begin{equation}
n_\varepsilon\rightarrow n ~~\mbox{a.e.}~~ \mbox{in}~~ \Omega\times (0,\infty)
\label{zjscz2.5297x9630222222}
\end{equation}
holds.

Due to these convergence properties (see \dref{zjscz2.5297x9630222222ee}--\dref{zjscz2.5297x9630222222}),
applying standard arguments we may take $\varepsilon = \varepsilon_j\searrow0$ in each term of the natural weak formulation of \dref{1.1fghyuisda} separately
to verify that in fact $(n,c,u)$ can be complemented by some pressure function $P$ in such a way that $(n,c,u,P)$ is a weak solution of (\ref{0.4})-(\ref{0.5}).
In the end, we can infer from the boundedness of $(n_{\varepsilon},c_{\varepsilon},u_{\varepsilon})$ and the Banach-Alaoglu theorem that $(n,c,u)$ is bounded.
\hfill$\Box$

{\bf Acknowledgement}:
This work is partially supported by  the Shandong Provincial
Science Foundation for Outstanding Youth (No. ZR2018JL005) and the National Natural
Science Foundation of China (No. 11601215). %and


\begin{thebibliography}{00}


\bibitem{Bellomo1216} N. Bellomo,  A. Belloquid, Y. Tao, M. Winkler,
\textit{Toward a mathematical theory of Keller--Segel models of pattern formation in biological tissues},
Math. Models Methods Appl. Sci., 25(9)(2015), 1663--1763.



\bibitem{Francesco12186} M. Di Francesco, A. Lorz, P. Markowich,
\textit{Chemotaxis--fluid coupled model for swimming bacteria with nonlinear diffusion: global existence and asymptotic behavior},
Discrete Contin. Dyn. Syst., 28(2010), 1437--1453.

\bibitem{Fujiwara66612186}  D. Fujiwara, H. Morimoto,  \textit{An $L^r$-theorem of the Helmholtz decomposition of vector fields},
J. Fac. Sci. Univ. Tokyo, 24(1977), 685-700.


\bibitem{Gilbarg4441215} D. Gilbarg, N.S. Trudinger,
\textit{Elliptic Partial Differential Equations of Second Order},
Springer, New York, 1983.

\bibitem{Haroske} D. D. Haroske, H. Triebel,
\textit{Distributions, Sobolev Spaces, Elliptic Equations, European Mathematical Society,}
Zurich, 2008.

\bibitem{Hillen} T. Hillen, K. Painter,
\textit{A user's guide to PDE models for chemotaxis,}
J. Math. Biol., 58(2009), 183--217.


\bibitem{Zhengssdddd00} Y. Ke, J. Zheng,
\textit{An optimal result for global existence and boundedness in a three-dimensional Keller-Segel(-Navier)-Stokes system (involving a tensor-valued sensitivity with saturation)},
arXiv:1806.07067.

\bibitem{Keller2710} E. Keller, L. Segel,
\textit{Model for chemotaxis},
J. Theor. Biol., 30(1970), 225--234.

\bibitem{Liggghh793} X. Li, Y. Wang, Z. Xiang,
\textit{Global existence and boundedness in a 2D Keller--Segel--Stokes system with nonlinear diffusion and rotational flux},
Commun. Math. Sci., 14(2016), 1889--1910.


\bibitem{Liuddfffff} J. Liu, Y. Wang,
\textit{Boundedness and decay property in a three-dimensional Keller-Segel-Stokes system involving tensor-valued sensitivity with saturation},
J. Diff. Eqns., 261(2)(2016), 967--999.


\bibitem{Painter55677} K. Painter, T. Hillen,
\textit{Volume-filling and quorum-sensing in models for chemosensitive movement},
Can. Appl. Math. Q. 10(2002), 501--543.


\bibitem{Peng55667} Y. Peng, Z. Xiang,
\textit{Global existence and boundedness in a 3D Keller--Segel--Stokes system with nonlinear diffusion and rotational flux},
Z. Angew. Math. Phys., (2017), 68:68.


\bibitem{Simon} J. Simon,
\textit{Compact sets in the space $L^{p}(O, T;B)$},
Annali di Matematica Pura ed Applicata, 146(1)(1986), 65--96.



\bibitem{Sohr} H. Sohr,
\textit{The Navier--Stokes equations, An elementary functional analytic approach},
Birkh\"{a}user Verlag, Basel (2001).


\bibitem{Tao794} Y. Tao, M. Winkler,
\textit{Boundedness in a quasilinear parabolic--parabolic Keller--Segel system with subcritical sensitivity},
J. Diff. Eqns., 252(2012), 692--715.

\bibitem{Tao41215} Y. Tao, M. Winkler,
\textit{Boundedness and decay enforced by quadratic degradation in a three-dimensional chemotaxis--fluid system},
Z. Angew. Math. Phys., 66(2015), 2555--2573.

\bibitem{Tuval1215} I. Tuval, L. Cisneros, C. Dombrowski, et al.,
\textit{Bacterial swimming and oxygen transport near contact lines},
Proc. Natl. Acad. Sci. USA, 102(2005), 2277--2282.

\bibitem{Wang23421215} Y. Wang, M. Winkler, Z. Xiang,
\textit{Global classical solutions in a two-dimensional chemotaxis-Navier-Stokes system with subcritical sensitivity},
Annali della Scuola Normale Superiore di Pisa-Classe di Scienze. XVIII, (2018), 2036--2145.


\bibitem{Wang21215} Y. Wang, Z. Xiang,
\textit{Global existence and boundedness in a Keller--Segel--Stokes system involving a tensor-valued sensitivity with saturation},
J. Diff. Eqns., 259(2015), 7578--7609.



\bibitem{Winkler79} M. Winkler,
\textit{Does a volume-filling effect always prevent chemotactic collapse},
Math. Methods Appl. Sci., 33(2010), 12--24.

\bibitem{Winkler11215} M. Winkler,
\textit{Boundedness and large time behavior in a three-dimensional chemotaxis--Stokes system with nonlinear diffusion and general sensitivity},
Calculus of Variations and Partial Diff. Eqns., 54(2015), 3789--3828.


\bibitem{Winkler51215} M. Winkler,
\textit{Global weak solutions in a three-dimensional chemotaxis-Navier-Stokes system},
Ann. Inst. H. Poincar\'{e} Anal. Non Lin\'{e}aire, 33(5)(2016),  1329---1352.


\bibitem{Xuess1215} C. Xue,
\textit{Macroscopic equations for bacterial chemotaxis: integration of detailed biochemistry of cell signaling},
J. Math. Biol. 70(2015), 1-44.

\bibitem{Xue1215} C. Xue, H. G. Othmer,
\textit{Multiscale models of taxis-driven patterning in bacterial population},
SIAM J. Appl. Math., 70(2009), 133--167.


\bibitem{Zhengsdsd6} J. Zheng,
\textit{Boundedness in a three-dimensional chemotaxis--fluid system involving tensor-valued sensitivity with saturation},
J. Math. Anal. Appl., 442(1)(2016), 353--375.



\bibitem{Zhenddddgssddsddfff00} J. Zheng,
\textit{An optimal result for global existence and boundedness in a three-dimensional Keller-Segel-Stokes system with nonlinear diffusion},
to appear, J. Diff. Eqns..




\end{thebibliography}
\end{document}